\documentclass[12pt]{amsart}
\usepackage{amssymb}
\usepackage[all]{xy}
\usepackage{lscape}
\usepackage{leftidx}

\theoremstyle{plain} 
\newtheorem{theor}[equation]{Theorem}
\newtheorem{cor}[equation]{Corollary}
\newtheorem{lem}[equation]{Lemma}
\newtheorem{conjecture}[equation]{Conjecture}
\newtheorem{proposition}[equation]{Proposition}

\theoremstyle{definition}
\newtheorem{defin}[equation]{Definition}

\theoremstyle{remark}
\newtheorem{rem}[equation]{Remark}
\newtheorem{ex}[equation]{Example}

\def\build#1_#2^#3{\mathrel{\mathop{\kern0pt#1}\limits_{#2}^{#3}}}

\begin{document}
\begin{abstract}
Let $M$ be a connected, closed  oriented manifold.
Let $\omega\in H^m(M)$ be its orientation class.
Let $\chi(M)$ be its Euler characteristic.
Consider the free loop fibration
$\Omega M\buildrel{i}\over\hookrightarrow LM\buildrel{ev}\over\twoheadrightarrow M$.
For any class $a\in H^*(LM)$ of positive degree, we prove that the cup product
$\chi(M)a\cup ev^*(\omega)$ is null.
In particular, if $i^*:H^*(LM;\mathbb{F}_p)\twoheadrightarrow
H^*(\Omega M;\mathbb{F}_p)$ is onto then
$\chi(M)$ is divisible by $p$ (or $M$ is a point).
\end{abstract}
\title{\bf String Topology, Euler Class and TNCZ free loop fibrations.}

\author{Luc Menichi}
\address{Larema, UMR CNRS 6093\\
Universit\'e d'Angers\\
2 Bd Lavoisier\\49045 Angers, FRANCE}
\email{firstname.lastname at univ-angers.fr}
\subjclass{ 55P50, 55R40}
\keywords{String Topology, free loop space, Euler Class}
\maketitle
\section{Introduction}
Denote by $LM:=map(S^1,M)$ the free loop space on $M$.
Except where specified, we work over an arbitrary principal ideal
domain ${\Bbbk}$.

In String Topology, shriek maps are used to defined operations.
But usually shriek maps are used to obtained vanishing results: see for example~\cite[III.10.1]{Brown:cohgro}
for an application of the transfer map in group cohomology.
In this paper, after defining them carefully, we use the operations in String Topology to obtain the following vanishing result:
\begin{theor}(Theorem~\ref{formule loop coproduct} 3) and Remark~\ref{noyau section egale cohomologie relative lacets libres lacets constants} below)
Let $M$ be a connected, closed  oriented manifold.
Let $\omega\in H^m(M)$ be its orientation class.
Let $\chi(M)$ be its Euler characteristic.
The cohomology of the free loops relative to the constant loops $H^*(LM,M)$ satisfies
$$
H^*(LM,M)\cup   \chi(M) ev^*(\omega)=\{0\}.
$$
\end{theor}
This vanishing result also holds for any generalized cohomology $h^*$ and homotopy fibre product of (the pull-back) of
an embedding with itself (Theorem~\ref{formule open string coproduct} 4) below).

Using Leray-Hirsch theorem or Serre spectral sequence, we deduce
\begin{cor}(Corollary~\ref{tncz implique euler nulle} below)
Let $M$ be a connected, closed  oriented manifold.
Suppose that the free loop fibration
$\Omega M\buildrel{i}\over\hookrightarrow LM\buildrel{ev}\over\twoheadrightarrow M$ is Totally Non-Cohomologous to Zero
with respect to a field $\mathbb{F}$, i. e. $H^*(i;\mathbb{F}):H^*(LM;\mathbb{F})\twoheadrightarrow
H^*(\Omega M;\mathbb{F})$  is onto.
Then $\chi(M)=0$ in $\mathbb{F}$ or $M$ is a point.
\end{cor}
Again this Corollary is generalized for any generalized cohomology $h^*$ and any fibration with section
(Lemmas~\ref{tncz implique classe d'euler nulle cohomologie generalisee} and~\ref{tncz implique classe d'euler nulle cohomologie singuliere} below).
We deduce then the following theorem.
\begin{theor} (Theorem~\ref{pull-back tncz implique cohomologie fibre ou classe d'euler nulle})
Let $g:G\hookrightarrow E$ be the pull-back of an embedding in the sense of definition~\ref{pull-back fibre ou transverse d'un embedding}.
Under some mild hypothesis, if the fibration $p_g$ associated to $g$ is Totally Non-Cohomologous to Zero and if all the homotopy fibres $p_g^{-1}(*)$ are not acyclic then the Euler class of $g$
is null.
\end{theor}
In the case of the diagonal embedding, Theorem 3 gives Corollary 2.

We now give the plan of the paper.

Part 1. We construct carefully the shriek maps used in String Topology to define the operations.
In particular, we give the key property (Proposition~\ref{shriek et Euler class})
that we use in this paper.

Part 2. We give our most general results. In section 5, we define the open string product and
open string coproduct of Sullivan~\cite{Sullivan:openclosedstring}. In section 6,
we compute the open string coproduct of the homotopy fibre product of an embedding with itself.
In section 7, we give general results on Totally Non-Cohomologous to Zero fibrations with sections
using Leray-Hirsch (Lemma~\ref{tncz implique classe d'euler nulle cohomologie generalisee}) or Serre spectral sequence (Lemma~\ref{tncz implique classe d'euler nulle cohomologie singuliere}).
These general results are used to prove Theorem 3.
In section 8, as an example, we consider the case when the embedding is the inclusion of
complex projective spaces.

Part 3. We specialize to the case of free loop spaces where the embedding is the diagonal embedding.
In section 9, we define the Chas-Sullivan loop product~\cite{Chas-Sullivan:stringtop} and the loop coproduct.
In section 10, we compute the dual of the loop coproduct in term of cup product. In particular,
we recover the results of Tamanoi~\cite{tamanoi-2007} and Sullivan~\cite{Sullivan:openclosedstring}
concerning the vanishing of the loop coproduct.
In section 11, we give (Theorem~\ref{Applications}) a variant of Theorem 1.
As application, we prove an homotopy version of a classical result relating fixed point action of
the circle and Euler characteristics. And we prove Corollary 2.
In section 12, we give many examples showing that Corollary 2 is pertinent.
We conjecture that Corollary 2 holds for any simply-connected finite CW-complex.
In section 13, we consider the case the case of relative free loop spaces. This is an example
where the embedding is the pull-back of the diagonal embedding.
As an application, we generalize Theorem 1 to the space $\text{map}(\vee_n S^1,M)$ of maps
from the wedge of $n$ circles to $M$.

{\em Acknowledgment:}
We wish to thank our expert in generalized cohomology, Geoffrey powell.
We would like also to thank the universities of Angers and Nantes for giving us the opportunity
to teach a Master course on characteristics classes following~\cite{Milnor-Stasheff}.

\part{The shriek maps}
\section{The shriek map of an oriented embedding}\label{shriek d'un embedding}
Let $h^*$ be a generalized cohomology theory which is additive and multiplicative.
 
Let $\phi:M\hookrightarrow B$ an embedding between two manifolds
without boundary of dimensions $m$
and $b$ respectively.
Following ~\cite[Corollary 11.2]{Milnor-Stasheff}, we suppose that $\phi(M)$ is a closed subset of $B$,
i. e.~\cite[Proposition A.53 (c), Theorem A.57]{Lee:Introsmoothmani} $\Phi$ is proper. Of course, this is the case if $M$ is compact.

We also suppose that the normal bundle $\nu$ is $h^*$-oriented.
If $M$ and $B$ are both $h^*$-oriented, $\nu$ is $h^*$-oriented
since $TM$ and $TB_{|M}=TM\oplus \nu$ are $h^*$-oriented (\cite[Chapter 4, Lemma 4.1]{Hirsch} or~\cite[Theorem 6 p. 45]{dyer:cohtheo}).

By the tubular neighborhood theorem (~\cite[11.1]{Milnor-Stasheff} or~\cite[Chapter 4, Theorem 5.2]{Hirsch}),
there exists an open neighborhood $V$ of $M$ in $B$ and a diffeomorphism
$exp:\nu\buildrel{\cong}\over\rightarrow V$ such that under this diffeomorphism, the zero
section map $M\rightarrow \nu$ corresponds to the inclusion map $s:M\hookrightarrow V$.

Consider the associated closed disk bundle $D(\nu)$ and the associated sphere bundle $S(\nu)$.
Let $N:=exp(D(\nu))$ be a closed tubular neighborhood. Let $\partial N:=exp(S(\nu))$ be its boundary.
Note that the inclusion map $s:M\buildrel{\approx}\over\hookrightarrow N$
is a homotopy equivalence.

Since $\nu$ is $h^*$-oriented, there exists a Thom class $u\in h^{b-m}(D(\nu);S(\nu))$ and a Thom
isomorphism~\cite[Theorem 9.1]{Milnor-Stasheff}.
\begin{rem}\label{Thom pas toujours iso}
Our generalized cohomology $h^*$ does not necessarily
satisfies the weak equivalence axiom. Therefore the Thom homomorphism might not be an isomorphism~\cite[(17.9.1)]{Dieck:algtop}.
\end{rem}
The Thom class $u$
will be thought as an element of  $h^{b-m}(N;\partial N)$.

Since $M$ is closed in $B$, $(B,B-M,N)$ is an excisive triad.
The inclusion $\partial N\buildrel{\approx}\over\hookrightarrow N-M$ is a homotopy equivalence.
Therefore the composite
$$
i:(N,\partial N)\rightarrow (N,N-M)\rightarrow (B,B-M)
$$
induces an isomorphism in cohomology.
Let $j:B\rightarrow (B,B-M)$ be the canonical map.
By definition~\cite[p. 419]{Felix-Thomas:stringtopGorenstein}, 
$\phi^!$ is the composite

\xymatrix@1{
h^*(M)\ar[r]^{s^{*-1}}_\cong
&h^*(N)\ar[r]^{-\cup u}_\cong
&h^{*+b-m}(N,\partial N)\ar[r]^{i^{*-1}}_\cong
&h^{*+b-m}(B,B-M)\ar[r]^{j^*}
&h^{*+b-m}(B)
}
\section{The shriek map of the pull-back of an embedding}\label{shriek d'un pull-back d'un embedding}
The idea to construct the shriek map of the pull-back of an embedding, is to pull-back the Thom class and
the tubular neighborhood. In particular, we will forget the original vector bundle $\nu$
and the fact the Thom homomorphism was (may-be see Remark~\ref{Thom pas toujours iso}) a Thom isormorphism.

Consider a (Serre) fibration $p:E\twoheadrightarrow B$.
Consider the pull-back diagram
$$\xymatrix{
\tilde{M}\ar[r]^{\tilde{\phi}}\ar[d]_{q}
&E\ar[d]^p\\
M\ar[r]_\phi
&B
}
$$
The goal of this section is to construct a shriek map for $\tilde{\phi}$.

Let $\tilde{N}:=p^{-1}(N)$. Then $p^{-1}(N-M)=p^{-1}(N)-p^{-1}(M)=\tilde{N}-\tilde{M}$.
Let $\tilde{\partial N}:=p^{-1}(\partial N)$.
Consider the two rectangles where all the squares are pull-backs

$$
\xymatrix{
\tilde{M}\ar[r]_{\tilde{s}}^\simeq\ar[d]_{q}
&\tilde{N}\ar[r]\ar[d]_p
&E\ar[d]^p\\
M\ar[r]_s^\approx
&N\ar[r]
&B
}
\quad\quad\quad
\xymatrix{
\tilde{\partial N}\ar[r]^\simeq\ar[d]_{p}
&\tilde{N}-\tilde{M}\ar[r]\ar[d]
&E\ar[d]^p\\
\partial N\ar[r]^\approx
&N-M\ar[r]
&B
}
$$
Since the inclusion map $s:M\buildrel{\approx}\over\hookrightarrow N$ is a homotopy equivalence and $p:\tilde{N}\twoheadrightarrow N$ is a (Serre) fibration,
 $\tilde{s}:\tilde{M}\buildrel{\simeq}\over\hookrightarrow \tilde{N}$ is a (weak) homotopy equivalence.
Similarly, since the inclusion $\partial N\buildrel{\approx}\over\hookrightarrow N-M$ is a homotopy equivalence,
the inclusion $\tilde{\partial N}\buildrel{\simeq}\over\hookrightarrow \tilde{N}-\tilde{M}$ is also
a (weak) homotopy equivalence.

Since the inverse image of an excisive triad is an excisive triad,
 $(E,E-\tilde{M},\tilde{N})=(p^{-1}(B),p^{-1}(B-M),p^{-1}(N))$ is an excisive triad.
Therefore the composite
$$
\tilde{i}:(\tilde{N},\tilde{\partial N})\rightarrow (\tilde{N},\tilde{N}-\tilde{M})\rightarrow (E,E-\tilde{M})
$$
induces an isomorphism in cohomology.

Let $\tilde{u}$ be the image of the Thom class $u$ by
$
p^*:h^*(N,\partial N)\rightarrow h^*(\tilde{N},\tilde{\partial N})
$.
Let $\tilde{j}:E\rightarrow (E,E-M)$ be the canonical map.
By definition, $\tilde{\phi}^!$ is the composite

\xymatrix@1{
h^*(\tilde{M})\ar[r]^{\tilde{s}^{*-1}}_\cong
&h^*(\tilde{N})\ar[r]^{-\cup \tilde{u}}
&h^{*+b-m}(\tilde{N},\tilde{\partial N})\ar[r]^{\tilde{i}^{*-1}}_\cong
&h^{*+b-m}(E,E-\tilde{M})\ar[r]^{\tilde{j}^*}
&h^{*+b-m}(E)
}
Comparing with the definition of the shriek map of $\phi$ given
in Section~\ref{shriek d'un embedding}, since $\tilde{u}:=p^*(u)$,
we obviously have the naturality with respect to pull-backs:
\begin{equation}\label{naturalite par rapport aux produits fibres}
p^*\circ \phi^!=\tilde{\phi}^!\circ q^*.
\end{equation}

Till now, our construction of the shriek map $\tilde{\phi}$ follows
the construction of Tamanoi in the special case of the loop
coproduct~\cite{tamanoi:capproducts} and of the loop coproduct~\cite{tamanoi-2007}
in string topology, except that Tamanoi, in each case, construct a specific
homotopy equivalence
$\tilde{N}\buildrel{\approx}\over\rightarrow\tilde{M}$ replacing our homotopy
equivalence $\tilde{s}:\tilde{M}\buildrel{\approx}\over\rightarrow\tilde{N}$.
As remarked by Tamanoi~\cite[p. 8]{tamanoi:capproducts},
note that in order to define $\tilde{\phi}$, we don't need to know
if the total space $q^*(S(\nu))$ of the bundle induced by pulling-back $S(\nu)$,
is diffeomorphic~\cite[Proposition 5.3]{Stacey:tdiftopoloops},
homeomorphic~\cite[p. 8]{Almeria}, or homotopy equivalent to
$\tilde{\partial N}$.

Although, we don't need it in this note, let us prove that
 $q^*(S(\nu))$ is homotopy equivalent to
$\tilde{\partial N}$ for completeness:
\begin{proposition}
Let $q^*(D(\nu))$ and $q^*(S(\nu))$ the pull-backs of the closed disk bundle
$D(\nu)$ and of the sphere disk bundle $S(\nu)$ along the (Serre) fibration
$q:\tilde{M}\twoheadrightarrow M$. Then there exist a (weak) homotopy equivalence
$$
\tilde{exp}:q^*(D(\nu))\buildrel{\simeq}\over\rightarrow\tilde{N}
$$
whose restriction to  $q^*(S(\nu))$
$$
\tilde{exp}:q^*(S(\nu))\buildrel{\simeq}\over\rightarrow\tilde{\partial N}
$$
is also a (weak) homotopy equivalence.
\end{proposition}
\begin{proof}
Since the bundle projection $\nu:D(\nu)\buildrel{\approx}\over\twoheadrightarrow M$ is a homotopy inverse to the zero section map $M\rightarrow D(\nu)$,
the following triangle commutes up to an homotopy
$H:[0,1]\times D(\nu)\rightarrow N$.
$$
\xymatrix{
D(\nu)\ar[d]^{\approx}_{\nu}\ar[r]^{exp}_\cong
&N\\
M\ar[ur]_s
}$$
The restriction of $H$ to $[0,1]\times S(\nu)$ is a homotopy
between the composite of $exp$ with the inclusion map,
$
S(\nu)\build\rightarrow_\cong^{exp}\partial N\hookrightarrow N
$,
and the composite
$
S(\nu)\buildrel{\nu}\over\twoheadrightarrow M\buildrel{s}\over\hookrightarrow N
$.
Therefore the two spaces $\tilde{S(\nu)}$ and $q^*(S(\nu))$ obtained
by pulling back this two composites along the (Serre) fibration
$p:\tilde{N}\twoheadrightarrow N$
are (weakly) homotopy equivalent~\cite[Chap. 2 Theorem 14]{Spanier:livre},

$$
\xymatrix{
\tilde{S(\nu)}\ar[r]^\cong\ar[d]_{p}
&\tilde{\partial N}\ar[r]\ar[d]_{p}
&\tilde{N}\ar[r]\ar[d]_p
&E\ar[d]^p\\
S(\nu)\ar[r]_{exp}^\cong &\partial N\ar[r]
&N\ar[r]
&B}
\quad
\xymatrix{
q^*(S(\nu))\ar[r]\ar[d]_q
&\tilde{M}\ar[r]^{\tilde{s}}_\simeq\ar[d]_{q}
&\tilde{N}\ar[r]\ar[d]_p
&E\ar[d]^p\\
S(\nu)\ar[r]_{\nu}
&M\ar[r]_s^\approx
&N\ar[r]
&B}
$$
Denote by $\tilde{exp}:q^*(S(\nu))\buildrel{\simeq}\over\rightarrow
\tilde{S(\nu)}\buildrel{\cong}\over\rightarrow\tilde{\partial{N}}$
the composite of the weak homotopy equivalence and of the homeomorphism.

Similarly, the homotopy $H$ gives a weak homotopy equivalence
$q^*(D(\nu))\buildrel{\simeq}\over\rightarrow\tilde{N}$.
Since $\tilde{exp}:q^*(S(\nu))\buildrel{\simeq}\over\rightarrow
\tilde{\partial{N}}$ was defined using the restriction of
$H$ to $S(\nu)$, we claim that this weak homotopy equivalence
$q^*(D(\nu))\buildrel{\simeq}\over\rightarrow\tilde{N}$
extends $\tilde{exp}:q^*(S(\nu))\buildrel{\simeq}\over\rightarrow
\tilde{\partial{N}}$. Therefore, we call it also  $\tilde{exp}$.
\end{proof}
If $p:E\twoheadrightarrow B$ is a fiber bundle, then using~\cite[4.6.4]{Aguilar-Gitler-Prieto},
we have that $q^*(S(\nu))$ and $\tilde{\partial N}$
are homeomorphic, instead of just homotopy equivalent.
\begin{rem}\label{vrai Thom class}
Since $q^*(S(\nu))\twoheadrightarrow S(\nu)$ and $p:\tilde{S(\nu)}\twoheadrightarrow S(\nu)$
are fiber homotopy equivalent~\cite[Chap. 2 Theorem 14]{Spanier:livre},
the induced isomorphism in cohomology
$$\tilde{exp}:h^*(\tilde{N},\tilde{\partial N})\buildrel{\cong}\over\rightarrow
h^*(q^*(D(\nu)),q^*(S(\nu)))
$$
fits into the commutative square
$$
\xymatrix{
h^*(\tilde{N},\tilde{\partial N})\ar[r]^{\tilde{exp}^*}_\cong
&h^*(q^*(D(\nu)),q^*(S(\nu)))\\
h^*(N,\partial N)\ar[r]^{exp^*}_\cong\ar[u]^{p^*}
&h^*(D(\nu),S(\nu))\ar[u]_{q^*}
}
$$
Therefore, by naturality of the Thom
class~\cite[11.7.11]{Aguilar-Gitler-Prieto},

-$\tilde{u}$, which was defined as $p^*(u)$, coincides with the Thom
class of the vector bundle $q^*(\nu)$ induced by pulling-back $\nu$
along $q:\tilde{M}\twoheadrightarrow M$ and

-the composite

$$
h^*(\tilde{M})\build\rightarrow_\cong^{\tilde{s}^{*-1}}h^*(\tilde{N})
\build\rightarrow_\cong^{-\cup\tilde{u}}h^{*+b-m}(\tilde{N},\tilde{\partial N})
$$ 
is a Thom isomorphism (if $h^*$ satisfies the weak equivalence axiom, see Remark~\ref{Thom pas toujours iso}).
\end{rem}
\section{The Euler class}
Let $\tilde{s}_{rel}:\tilde{M}\rightarrow (\tilde{N},\tilde{\partial N})$
be the relative inclusion map. Since, by Remark~\ref{vrai Thom class},
$\tilde{u}$ is the Thom class of the vector bundle  $q^*(\nu)$ obtained by pull-back, $\tilde{s}_{rel}^*(\tilde{u})$ is its Euler class.
\begin{proposition}\label{shriek et Euler class}~\cite[Theorem 2.1 (5)]{Tamanoi:TQFTopenclosed}
For the shriek map of the pull-back of an embedding, we have the formula
for any $x\in h^*(\tilde{M})$:
$$
\tilde{\phi}^*\circ\tilde{\phi}^!(x)=x\cup \tilde{s}_{rel}^*(\tilde{u}).
$$
\end{proposition} 
For an embedding $\phi:M\hookrightarrow B$, this formula is well known
~\cite[Theorem 6.1 (5)]{Kawakubo:transformationgroups}.
For $\tilde{\phi}$, the pull-back of an embedding, the proof will be similar~\cite[p. 282]{Kawakubo:transformationgroups}.
\begin{proof}
Remark that the following square commutes
$$
\xymatrix{
(\tilde{N},\tilde{\partial N})\ar[r]^{\tilde{i}}
&(E,E-\tilde{M})\\
\tilde{M} \ar[r]_{\tilde{\phi}}\ar[u]^{\tilde{s}_{rel}}
& E.\ar[u]_{\tilde{j}}
}
$$
Remark also that
$$\tilde{s}_{rel}^*:h^*(\tilde{N},\tilde{\partial N})\rightarrow h^*(\tilde{M})
$$
is $h^*(\tilde{N})$-linear where $h^*(\tilde{N})$ acts on $h^*(\tilde{M})$
by restriction of scalar with respect to the algebra morphism
$\tilde{s}^*:h^*(\tilde{N})\rightarrow h^*(\tilde{M})
$. Therefore, by definition of $\tilde{\phi}^!$,
\begin{multline*}
\tilde{\phi}^*\circ\tilde{\phi}^!(x)
=\tilde{\phi}^*\circ\tilde{j}^*\circ\tilde{i}^{*-1}
(\tilde{s}^{*-1}(x)\cup\tilde{u})\\
=\tilde{s}_{rel}^*\circ \tilde{i}^{*}\circ \tilde{i}^{*-1}
(\tilde{s}^{*-1}(x)\cup\tilde{u})\\
=\left(\tilde{s}^*\circ \tilde{s}^{*-1}(x)\right)\cup \tilde{s}_{rel}^*(\tilde{u})
=x\cup \tilde{s}_{rel}^*(\tilde{u})
\end{multline*}
\end{proof}
Consider the commutative diagram
$$
\xymatrix{
\tilde{M}\ar[r]^{\tilde{s}_{rel}}\ar[d]_q
&(\tilde{N},\tilde{\partial N})\ar[d]^p\\
M\ar[r]^{s_{rel}}
&(N,\partial N)
}
$$
Since, by definition, $\tilde{u}:=p^*(u)$, we obviously have
\begin{equation}\label{naturalite classe d'Euler}
\tilde{s}_{rel}^*(\tilde{u})=q^*\circ s_{rel}^*(u).
\end{equation}
Note that, by definition (~\cite[p. 279]{Kawakubo:transformationgroups}
or~\cite[p. 98]{Milnor-Stasheff}), $s_{rel}^*(u)$ is the Euler class
of the normal bundle $\nu$.
Since we have remarked that $\tilde{s}_{rel}^*(\tilde{u})$
is the Euler class of the vector bundle obtained by pulling-back $\nu$,
this formula is just the naturality of the Euler
class~\cite[Property 9.2]{Milnor-Stasheff}.
\begin{cor}\label{shriek et inclusion fibre nulle}
Let $k:F\hookrightarrow E$ be the inclusion of the fiber of the fibration $p$.
If $b>m$ then $k^*\circ \tilde{\phi}^! =0$ in singular cohomology.
\end{cor}
\begin{proof}
Let $l:F\hookrightarrow \tilde{M}$ be the inclusion of the fiber of the induced fibration $q$.
We have the commutative diagram
$$\xymatrix{
F\ar[r]^l\ar[d]_\varepsilon\ar@/^2pc/[rr]^k
&\tilde{M}\ar[r]^{\tilde{\phi}}\ar[d]_{q}
&E\ar[d]^p\\
\{*\}\ar[r]_\eta
&M\ar[r]_\phi
&B}
$$
where the two squares are pull-backs.
By Proposition~\ref{shriek et Euler class} and equation~(\ref{naturalite classe d'Euler}),
$$
\tilde{\phi}^*\circ\tilde{\phi}^!(x)=x\cup q^*(e_\nu)
$$
where $e_\nu\in H^{b-m}(M)$ is the Euler class of the normal bundle $\nu$.
Therefore
$$
k^*\circ \tilde{\phi}^!(x)=l^*\circ\tilde{\phi}^*\circ\tilde{\phi}^!(x)=l^*(x)\cup l^*\circ q^*(e_\nu)=l^*(x)\cup \varepsilon^*\circ \eta^*(e_\nu)=0.
$$
\end{proof}

\part{The general case}
\section{The open string (co)products}
In this section, we define the joining product $\wedge$ of Sullivan~\cite{Sullivan:openclosedstring} that following Tamanoi~\cite{Tamanoi:TQFTopenclosed}, we prefer
to call the open string product.
And more important for us, we define a generalized version of the cutting at time $1/2$ product $\vee_{1/2}$ of Sullivan~\cite{Sullivan:openclosedstring} that following Tamanoi~\cite{Tamanoi:TQFTopenclosed}, we prefer
to call the open string coproduct.

Let $f:F\rightarrow E$ and $g:G\rightarrow E$ be two maps.
Let $$\leftidx{^f}{E}{^g}=\{(a,\omega,b)\in F\times E^I\times G/ f(a)=\omega(0),\;g(b)=\omega(1)\}$$
denote the {\em homotopy fibre product} of $f$ and $g$ which is obtained by the following pull-back
\xymatrix{
{^f}E^g\ar[r]\ar[d]_{(ev_0,ev_1)}
&E^I\ar[d]^{(ev_0,ev_1)}\\
F\times G\ar[r]_-{(f\times g)}
&E\times E.
}
\begin{defin}\label{pull-back fibre ou transverse d'un embedding}
A continuous map $g:G\rightarrow E$ is the {\em pull-back of an embedding}
if it is equipped with a pull-back diagram
$$\xymatrix{
G\ar[r]^{g}\ar[d]_{q}
&E\ar[d]^p\\
M\ar[r]_\phi
&B
}
$$
where $\phi:M\hookrightarrow B$ a proper embedding between two manifolds of codimension $m$
with $h^*$-oriented normal bundle
and where the map $p$ is a (Serre) fibration (This is exactly the case considered in
Section~\ref{shriek d'un pull-back d'un embedding}).
\end{defin}
\begin{rem}\label{pull-back transverse d'un embedding}
Suppose that a continous map $g:G\rightarrow E$ fits into a 
pull-back diagram
$$\xymatrix{
G\ar[r]^{g}\ar[d]_{q}
&E\ar[d]^p\\
M\ar[r]_\phi
&B
}
$$
where $\phi:M\hookrightarrow B$ a proper embedding between two manifolds of codimension $m$
with $h^*$-oriented normal bundle
and where the map $p$ is smooth and is transverse to $\phi$.

Then $g:G\hookrightarrow E$ a proper embedding between two manifolds of codimension $m$
with $h^*$-oriented normal bundle.
\end{rem}
A common example to both definition~\ref{pull-back fibre ou transverse d'un embedding} and remark~\ref{pull-back transverse d'un embedding} is when $p$ is a smooth fibre bundle.
\begin{proof}[Proof of Remark~\ref{pull-back transverse d'un embedding}]
Since $\phi$ and $p$ are transverse,
$g:G\hookrightarrow E$ is an embedding of codimension $m$.
Note that the normal bundle of $g$, $\nu_{g}$, is the pull-back $q^*(\nu_\phi)$ of the normal bundle of $\phi$, $\nu_\phi$, along $q$.
Therefore since $\nu_\phi$ is $h^*$-oriented, $\nu_{g}$ is also
$h^*$-oriented~\cite[11.7.11]{Aguilar-Gitler-Prieto} without assuming
that $G$ or $E$ is $h^*$-orientable.
Since $\phi(M)$ is a closed subset of $B$, $g(G)=p^{-1}(\phi(M))$ is a closed subset of $E$.
\end{proof}
\begin{theor}~\cite{Sullivan:openclosedstring}
Let $f:F\rightarrow E$, $g:G\rightarrow E$  and $h:H\rightarrow E$ be three maps.

\noindent 1) If $G$ is a smooth $h^*$-oriented manifold without boundary of dimension $n$ then there is a
open string product
$$lp_{f,g,h}:H_*({^f}E^g)\otimes H_*({^g}E^h)\rightarrow H_{*-n}({^f}E^h).$$

\noindent 2) Suppose that $g$ is the pull-back of an embedding in the
sense of definition~\ref{pull-back fibre ou transverse d'un embedding}.
Then there is a open string coproduct
$$lcp_{f,g,h}:h^*({^f}E^g)\otimes h^*({^g}E^h)\rightarrow h^{*+m}({^f}E^h).$$
\end{theor}
\begin{rem}
In~\cite{Sullivan:openclosedstring}, Sullivan consider the open string
product only in the case when $g$ is an embedding. 
\end{rem}
\begin{proof}
1) Consider the three pull-back squares (Compare with~\cite[Diagram p.11]{Tamanoi:TQFTopenclosed})
$$
\xymatrix{
{^f}E^g\times {^g}E^h\ar[d]_{ev_1\times ev_0}
& {^f}E^g\times_E {^g}E^h\ar[l]_{\tilde{\Delta}}\ar[r]^-{\mu_{f,g,h}}\ar[d]_{ev_{1/2}}
& {^f}E^h\ar[d]^{ev_{1/2}}\\
G\times G
& G\ar[l]^\Delta\ar[r]_{g}\ar[d]_{q}
&E\ar[d]_{p}\\
& M\ar[r]_\phi
&B
}
$$
The open string product $lp_{f,g,h}$ is defined as the Chas-Sullivan loop product above using the left pull-back square:
$$
H_*{^f}E^g\otimes H_*{^g}E^h\buildrel{\times}\over\rightarrow H_*{^f}E^g\times {^g}E^h
\buildrel{\tilde{\Delta}_!}\over\rightarrow H_{*-n}{^f}E^g\times_E {^g}E^h
\buildrel{\mu_{f,g,h*}}\over\rightarrow H_{*-n}{^f}E^h.
$$

2) Since  $p$ is a (Serre) fibration, then the composite
$$\xymatrix@1{{^f}E^h\ar[r]^{ev_{1/2}}
&E\ar[r]^{p}
&B
}$$
is also a fibration. Therefore using the total right rectangle, since $\phi:M\hookrightarrow B$ is an embedding, we obtain the shriek map
$$
\mu_{f,g,h}^!:h^*({^f}E^g\times_E {^g}E^h)\rightarrow h^{*+m}({^f}E^h).
$$
The open string coproduct $lcp_{f,g,h}$ is now defined as the composite
$$
h^*({^f}E^g)\otimes h^*({^g}E^h)\buildrel{\times}\over\rightarrow h^*({^f}E^g\times {^g}E^h)
\buildrel{\tilde{\Delta}^*}\over\rightarrow h^{*}({^f}E^g\times_E {^g}E^h)
\buildrel{\mu_{f,g,h}^!}\over\rightarrow h^{*+m}({^f}E^h).
$$
\end{proof}
\section{A simple formula for the open string coproduct}
\begin{theor}\label{formule open string coproduct}
Let $f:F\rightarrow E$, $g:G\rightarrow E$ and $h:H\rightarrow E$ be three maps.
Suppose that $g$ is the pull-back of an embedding in the sense of definition~\ref{pull-back fibre ou transverse d'un embedding}.
Let $e_{\nu}\in h^m(M)$ be the Euler class of the normal bundle of the embedding $\phi:M\hookrightarrow B$.

Let $ev_0: \leftidx{^g}{E}{^h}\twoheadrightarrow G$, $(a,w,b)\mapsto a$
be the first projection map.
Let $ev_1: \leftidx{^f}{E}{^g}\twoheadrightarrow G$, $(a,w,b)\mapsto b$
be the second projection map.

Let $\sigma:G\hookrightarrow \leftidx{^g}{E}{^g}$, $b\mapsto (b,\text{constant path }g(b),b)$ be the section of both projections map $ev_0$ and $ev_1$ when $f=g=h$.

Then 

\noindent 1) the open string coproduct
$$lcp_{f,g,g}:h^*(\leftidx{^f}{E}{^g})\otimes h^*(\leftidx{^g}{E}{^g})\rightarrow h^{*+m}(\leftidx{^f}{E}{^g})$$
is given by
$$lcp_{f,g,g}(a\otimes b)=a\cup ev_1^*\left(\sigma^*(b)\cup q^*(e_{\nu})\right).$$

\noindent 2) the open string coproduct
$$lcp_{g,g,h}:h^*(\leftidx{^g}{E}{^g})\otimes h^*(\leftidx{^g}{E}{^h})\rightarrow h^{*+m}(\leftidx{^g}{E}{^h})$$
is given by
$$lcp_{g,g,h}(b\otimes c)=ev_0^*\circ \sigma^*(b)\cup c\cup ev_0^*\circ q^*(e_{\nu}).$$

\noindent 3) In the case $f=g=h$,
$$ev_0^*\circ q^*(e_{\nu})=ev_1^*\circ q^*(e_{\nu})\in h^*(\leftidx{^g}{E}{^g}).$$

\noindent 4) The ideal $\text{Ker } \sigma^*:h^*(\leftidx{^g}{E}{^g})\twoheadrightarrow h^*(G)$ satisfies
$$
\text{Ker } \sigma^*\cup ev_0^*\circ q^*(e_{\nu})=\{0\}.
$$

\noindent 5) the open string coproduct on $h^*(\leftidx{^g}{E}{^g})$
$$lcp_{g,g,g}:h^*(\leftidx{^g}{E}{^g})\otimes h^*(\leftidx{^g}{E}{^g})\rightarrow h^{*+m}(\leftidx{^g}{E}{^g})$$
is given by
$$lcp_{g,g,g}(a\otimes b)=ev_0^*\left(\sigma^*(a)\cup \sigma^*(b) \cup q^*(e_{\nu})\right)
=ev_1^*\left(\sigma^*(a)\cup \sigma^*(b) \cup q^*(e_{\nu})\right)
.$$

\noindent 6) Let $\alpha\in\leftidx{^g}{E}{^g}$. Denote by $\leftidx{^g}{E}{^g}_{[\alpha]}$
the path-connected component of $\alpha$.
Denote by
$
\leftidx{^g}{E}{^g}_0:=\cup_{\alpha\in G} \leftidx{^g}{E}{^g}_{[\sigma(\alpha)]}
$, i. e. the subspace of $\leftidx{^g}{E}{^g}$, union of the path-connected components $\text{Im }\pi_0(\sigma):\pi_0(G)\hookrightarrow \pi_0(\leftidx{^g}{E}{^g})$
(If $G$ is path-connected, there is only one). Then

the open string coproduct 
$lcp_{f,g,g}$ is trivial outside of $h^*(\leftidx{^f}{E}{^g})\otimes h^*(\leftidx{^g}{E}{^g}_0)$,

the open string coproduct 
$lcp_{g,g,h}$ is trivial outside of $h^*(\leftidx{^g}{E}{^g}_0)\otimes h^*(\leftidx{^g}{E}{^h})$,

the open string coproduct 
$lcp_{g,g,g}$ is trivial outside of $h^*(\leftidx{^g}{E}{^g}_0)\otimes h^*(\leftidx{^g}{E}{^g}_0)$.

\noindent 7) The open string coproduct 
$lcp_{f,g,g}$ maps $h^*(\leftidx{^f}{E}{^g}_{[\alpha]})\otimes h^*(\leftidx{^g}{E}{^g})$ to $h^{*+m}(\leftidx{^f}{E}{^g}_{[\alpha]})$.

The open string coproduct 
$lcp_{g,g,h}$ maps $h^*(\leftidx{^g}{E}{^g})\otimes h^*(\leftidx{^g}{E}{^h}_{[\alpha]})$ to $h^{*+m}(\leftidx{^g}{E}{^h}_{[\alpha]})$.

The image of the open string coproduct $lcp_{g,g,g}$
is contained in $\leftidx{^g}{E}{^g}_0$.

\noindent 8) Suppose that there is an integer $\text{dim }G$
such that $\forall i>\text{dim }G$, $h^i(G)=\{0\}$.
Then for all $a$, $b\in h^*(\leftidx{^g}{E}{^g})$ such that $\vert a\vert+\vert b\vert>\text{dim }G-m$,
$lcp_{g,g,g}(a\otimes b)=0$.
\end{theor}
\begin{proof}

\noindent 1) and 2) Recall that we have the following two pull-back squares
$$
\xymatrix{
& {^f}E^g\times_E {^g}E^h\ar[r]^-{\mu_{f,g,h}}\ar[d]_{ev_{1/2}}
& {^f}E^h\ar[d]^{ev_{1/2}}\\
& G\ar[r]_{g}\ar[d]_{q}
&E\ar[d]_{p}\\
& M\ar[r]_\phi
&B
}
$$

Denote by $(q\circ ev_{1/2})^*(\nu)$ the pull-back of $\nu$ along $q\circ ev_{1/2}$.
Let $e_{(q\circ ev_{1/2})^*(\nu)}$ by its Euler class.
By Proposition~\ref{shriek et Euler class}, for any $x\in h^*(\leftidx{^f}{E}{^g}\times_E \leftidx{^g}{E}{^h})$
$$
\mu^*_{f,g,h}\circ\mu^!_{f,g,h}(x)=x\cup e_{(q\circ ev_{1/2})^*(\nu)}.
$$
By formula~(\ref{naturalite classe d'Euler}) (Naturality of Euler class),
$$
 e_{(g\circ q)^*(TM)}= (q\circ ev_{1/2})^*(e_{\nu})
$$
where $e_{\nu}$ is the Euler class of the normal bundle.

Putting everything together, we have that
\begin{equation}\label{mu et son shriek}
\mu^*_{f,g,h}\circ\mu^!_{f,g,h}(x)=    x\cup ev_{1/2}^*\circ q^*(e_{\nu}).
\end{equation}
Note that till here, the same discussion for the open string product,
shows that
\begin{equation}\label{delta et son shriek}
\tilde{\Delta}^*\circ\tilde{\Delta}^!(x)=    x\cup ev_{1/2}^*(e_{TG}).
\end{equation}
But now, we will see that both $\mu^*_{f,g,g}$ and $\mu^*_{g,g,h}$ admit a retract.
In particular, $\mu^*_{g,g,g}$ admits two different retracts $i_1^*$ and $i_2^*$.
Let $i_1:\leftidx{^f}{E}{^g} \rightarrow \leftidx{^f}{E}{^g}\times_E \leftidx{^g}{E}{^g}$ be the inclusion map on the first factor defined by
$$i_1(a,w,b)=\left((a,w,b),(b,\text{constant path }g(b),b)\right)$$
for $a\in F$, $w\in E^I$ and $b\in G$ such that $f(a)=w(0)$ and $g(b)=w(1)$.
Let $i_2:\leftidx{^g}{E}{^h} \rightarrow \leftidx{^g}{E}{^g}\times_E \leftidx{^g}{E}{^h}$ be the inclusion map on the second factor defined by
$$i_2(a,w,b)=\left((a,\text{constant path }g(a),a),(a,w,b))\right)$$
for $a\in G$, $w\in E^I$ and $b\in H$ such that $(g(a),h(b))=(w(0),w(1))$.
The first inclusion $i_1$ is a section up to homotopy of $\mu_{f,g,g}$.
Therefore $i_1^*$ is a retract of $\mu_{f,g,g}^*$.

So by formula~(\ref{mu et son shriek}), since $ev_{1/2}\circ i_1$ is the projection map
$ev_1:\leftidx{^f}{E}{^g}\twoheadrightarrow G$, $(a,w,b)\mapsto b$, for any $x\in h^*(\leftidx{^f}{E}{^g}\times_E \leftidx{^g}{E}{^g})$,
\begin{multline*}
\mu_{f,g,g}^!(x)=i_1^*\circ \mu_{f,g,g}^*\circ \mu_{f,g,g}^!(x)
=i_1^*(x\cup ev_{1/2}^*\circ q^*(e_{\nu}))\\
= i_1^*(x)\cup i_1^*\circ ev_{1/2}^*\circ q^*(e_{\nu}))
= i_1^*(x)\cup ev_1^*\circ q^*(e_{\nu}).
\end{multline*}
Similarly since $i_2^*$ is a retract of $\mu_{g,g,h}^*$
and since $ev_{1/2}\circ i_2$ is the projection map
$ev_0:\leftidx{^g}{E}{^h}\twoheadrightarrow G$, $(a,w,b)\mapsto a$,
for any $x\in h^*(\leftidx{^g}{E}{^g}\times_E \leftidx{^g}{E}{^h})$,
$$\mu_{g,g,h}^!(x)= i_2^*(x)\cup ev_0^*\circ q^*(e_{\nu}).$$
Since the following two squares commute
$$
\xymatrix{
\leftidx{^f}{E}{^g}\ar[r]^-{i_1}\ar[d]_\Delta
& \leftidx{^f}{E}{^g}\times_N \leftidx{^g}{E}{^g}\ar[d]^{\tilde{\Delta}}\\
\leftidx{^f}{E}{^g}\times \leftidx{^f}{E}{^g}\ar[r]_{id\times (\sigma\circ ev_1)}
&\leftidx{^f}{E}{^g}\times \leftidx{^g}{E}{^g}
}
\quad\quad
\xymatrix{
\leftidx{^g}{E}{^h}\ar[r]^-{i_2}\ar[d]_\Delta
& \leftidx{^g}{E}{^g}\times_N \leftidx{^g}{E}{^h}\ar[d]^{\tilde{\Delta}}\\
\leftidx{^g}{E}{^h}\times \leftidx{^g}{E}{^h}\ar[r]_{(\sigma\circ ev_0)\times id}
&\leftidx{^g}{E}{^g}\times \leftidx{^g}{E}{^h},
}
$$
for any $a\in h^*(\leftidx{^f}{E}{^g})$ and $b\in h^*(\leftidx{^g}{E}{^g})$,
\begin{multline*}
\mu_{f,g,g}^!\circ\tilde{\Delta}^*(a\times b)
= i_1^*\circ\tilde{\Delta}^*(a\times b)\cup ev_1^*\circ q^*(e_{\nu})\\
=a\cup (\sigma\circ ev_1)^*(b)\cup ev_1^*\circ q^*(e_{\nu})
\end{multline*}
and for any $b\in h^*(\leftidx{^g}{E}{^g})$ and $c\in h^*(\leftidx{^g}{E}{^h})$,
$$
\mu_{g,g,h}^!\circ\tilde{\Delta}^*(b\times c)
=(\sigma\circ ev_0)^*(b)\cup c\cup ev_0^*\circ q^*(e_{\nu}).
$$

\noindent 3) Using 1), $$ lcop_{g,g,g}(1\otimes 1)=
ev_1^*\circ q^*(e_{\nu}).$$

Using 2), $$ lcop_{g,g,g}(1\otimes 1)=
ev_0^*\circ q^*(e_{\nu}).$$

\noindent 4) Let $\varepsilon$ be $0$ or $1$.
Since $\sigma^*\circ ev_\varepsilon^*=id$, using the split short exact sequence
$$
0\rightarrow \text{Ker } \sigma^*\buildrel{i}\over\hookrightarrow h^*(\leftidx{^g}{E}{^g})\buildrel{\sigma^*}\over\twoheadrightarrow h^*(G)\rightarrow 0,
$$
we obtain that $\text{Ker } \sigma^*=\text{Im }r_\varepsilon$ where
$r_\varepsilon:h^*(\leftidx{^g}{E}{^g})\twoheadrightarrow \text{Ker } \sigma^*$ is the retract of the inclusion defined
by $r_\varepsilon(a)=a-ev_\varepsilon^*\circ\sigma^*(a)$.

By 1) and 3), $$ lcop_{g,g,g}(a\otimes 1)=a\cup  ev_1^*\circ q^*(e_{\nu})=
a\cup  ev_0^*\circ q^*(e_{\nu}).$$

By 2), $$ lcop_{g,g,g}(a\otimes 1)=ev_0^*\circ\sigma^*(a)\cup  ev_0^*\circ q^*(e_{\nu}).$$
Therefore $r_0(a)\cup  ev_0^*\circ q^*(e_{\nu})=0$.

\noindent 5) By 4), $$b\cup  ev_0^*\circ q^*(e_{\nu})=ev_0^*\circ\sigma^*(b)\cup  ev_0^*\circ q^*(e_{\nu}).$$
Therefore by 2) $$ lcop_{g,g,g}(a\otimes b)=ev_0^*\circ\sigma^*(a)\cup ev_0^*\circ\sigma^*(b)\cup  ev_0^*\circ q^*(e_{\nu}).$$

By 4) and 3) $$a\cup  ev_1^*\circ q^*(e_{\nu})=ev_1^*\circ\sigma^*(a)\cup  ev_1^*\circ q^*(e_{\nu}).$$
Therefore by 1) and the graded commutativity of the cup product
$$
 lcop_{g,g,g}(a\otimes b)=ev_1^*\circ\sigma^*(a)\cup ev_1^*\circ\sigma^*(b)\cup  ev_1^*\circ q^*(e_{\nu}).
$$
\noindent 6) Since by definition, $\sigma$ arrives inside $\leftidx{^g}{E}{^g}_0$, $\sigma^*$ factorizes through the projection $h^*(\leftidx{^g}{E}{^g})\twoheadrightarrow h^*(\leftidx{^g}{E}{^g}_0)$.
So using 1) and 2), 6) is proved.

\noindent 7)
Let $a\in h^*(\leftidx{^f}{E}{^g}_{[\alpha]})$ and $b\in h^*(\leftidx{^g}{E}{^g})$.
The cohomology of a space is isomorphic to the product of the cohomology
algebras of its path-connected components.
Therefore the cup product with $ev_1^*(\sigma^*(b)\cup q^*(e_\nu))$ defines a linear
application from $h^*(\leftidx{^f}{E}{^g}_{[\alpha]})$ to itself.
So by 1), $lcop_{f,g,g}(a\otimes b)\in h^*(\leftidx{^f}{E}{^g}_{[\alpha]})$.

In the case $f=g$, by 6),
if $lcop_{g,g,g}(a\otimes b)$ is non-zero then the path-connected component of $\alpha$ belongs to the image of $\pi_0(\sigma)$ and in this case
$lcop_{g,g,g}(a\otimes b)\in h^*(\leftidx{^f}{E}{^g}_0)$.

\noindent 8) The element $\sigma^*(a)\cup \sigma^*(b) \cup q^*(e_{\nu})\in h^*(G)$ is of degree
$\vert a\vert+\vert b\vert +m >\text{dim }G$ and so is null. Using 5), we obtain that
$
 lcop_{g,g,g}(a\otimes b)=0
 $.
\end{proof}
\section{TNCZ fibrations}
\begin{lem}\label{tncz implique classe d'euler nulle cohomologie generalisee}
Suppose that $B$ is path-connected. Suppose also that $B$ is a CW-complex or that $h^*$ satisfies the weak equivalence axiom.
Let $F\buildrel{i}\over\hookrightarrow E\buildrel{p}\over\twoheadrightarrow B$ be
a (Serre) fibration with base $B$ which admits a section $\sigma:B\rightarrow E$
up to homotopy, i. e. $p\circ \sigma\thickapprox id_B$ and an element $e\in h^*(B)$
such that
$$
\text{Ker }\sigma^*\cup p^*(e)=\{0\}.
$$
If the fibration $p$ is Totally Non-Cohomologous to Zero, i. e. $h^*(i):h^*(E)\twoheadrightarrow h^*(F)$ is onto and if $\tilde{h}^*(F)$ is a finitely generated free graded ${h^*}$-module
 then $e=0$ or $\tilde{h}^*(F)= \{0\}$.
\end{lem}
\begin{rem}\label{changement de fibres}
Let $p:E\twoheadrightarrow B$ be a  fibration. Let $\omega:I\rightarrow B$ be a path from $b$ to $b'$.
Let $w_{\#}:p^{-1}(b)\buildrel{\thickapprox}\over\rightarrow p^{-1}(b')$ be the induced homotopy equivalence
between the fibers~\cite[Theorem 2.8.12]{Spanier:livre}.
By definition, $w_{\#}$ commutes up to homotopy with the inclusions of fibers
$i_b:p^{-1}(b)\hookrightarrow E$ and $i_{b'}:p^{-1}(b')\hookrightarrow E$. In cohomology,
$i_b^*=w_{\#}^*\circ i_{b'}^*$~\cite[Proof of (17.9.3)]{Dieck:algtop}.
So $i_{b}^*$ is surjective if and only if $i_{b'}^*$ is surjective.
And given a family of vectors $c_j\in h^*(E)$, the $i_{b}^*(c_j)$'s form a ${h^*}$-basis of $h^*(p^{-1}(b))$
if and only if the $i_{b'}^*(c_j)$'s form a ${h^*}$-basis of $h^*(p^{-1}(b'))$.
\end{rem}
\begin{proof}
Let $b_0\in B$. Let $\varepsilon_F:F\rightarrow *$ be the unique map to a point.
Since $B$ is path-connected and $\tilde{h}^*(F):=\text{coker }\varepsilon_F^*:h^*\rightarrow h^*(F)$ does not
depend of a base point and is homotopy invariant~\cite[17.1.3]{Dieck:algtop}, we can chose any fibre $F$.
We take $F:=p^{-1}\left(\{p\circ\sigma(b_0)\}\right)$.
Since $h^*(i):h^*(E)\twoheadrightarrow h^*(F)$ is surjective, $\tilde{h}^*(i):\tilde{h}^*(E)\twoheadrightarrow \tilde{h}^*(F)$ is also surjective.

Since $\tilde{h}^*(\sigma)\circ \tilde{h}^*(p)=id$,
$
\tilde{h}^*(E)=\text{Ker }\tilde{h}^*(\sigma)\oplus \text{Im }\tilde{h}^*(p).
$
Since $\text{Im }\tilde{h}^*(p)\subset \text{Ker }\tilde{h}^*(i)$, the restriction of
$\tilde{h}^*(i)$ to $\text{Ker }\tilde{h}^*(\sigma)$ is also surjective.

Suppose that $\tilde{h}^*(F)\neq \{0\}$. Then 
there exists classes $c_j\in \text{Ker }\tilde{h}^*(\sigma)\subset\tilde{h}^*(E)$ such that
the $\tilde{h}^*(i)(c_j)$'s form a ${h^*}$-basis of $\tilde{h}^*(F)$.

Now let $b_0\in B$, $\sigma(b_0)\in E$ and $\sigma(b_0)\in F$ be the chosen base points $*$ of $B$, $E$ and $F$.
Denote by $\eta_{x_0}^X:\{x_0\}\hookrightarrow X$ be the inclusion of the base point $x_0$ into a based
space $X$. Then we have the canonical identification for a based space $X$ between $\tilde{h}^*(X)$
and $h^*(X,x_0)\cong \text{Ker } \eta_{x_0}^{X*}:h^*(X)\rightarrow h^*$.
Since $\sigma:(B,*)\hookrightarrow (E,*)$
and $i:(F,*)\hookrightarrow (E,*)$
are based maps, we can consider that the classes $c_j\in h^*(E,*)$, that
the  $i^*(c_{j})$'s form a $h^*$-basis of $h^*(F,*)$, that 
$\text{Ker }\tilde{h}^*(\sigma)=\text{Ker }\sigma^*$ and so that $c_{j}\cup p^*(e)=0$.

%

Since $h^*(F)\cong h^*(F,*)\oplus h^*$, the classes $c_j\in h^*(E,*)$
and the unit $1\in h^0(E)$ are send by $i^*$ to a $h^*$-basis of $h^*(F)$.
So by the Leray-Hirsch theorem for generalized cohomology (\cite[(17.8.4)]{Dieck:algtop} using Remark~\ref{changement de fibres}),  $h^*(E)$ is a free $h^*(B)$-modules with basis $1$ and the $c_j$'s. So $e=0$.
\end{proof}
\begin{rem}
In Lemma~\ref{tncz implique classe d'euler nulle cohomologie generalisee}, when the generalized cohomology $h^ *$ is a singular cohomology $H^*$, it is enough to suppose that $\tilde{H}^ q(F)$ is ${\Bbbk}$-free
module of finite type for each degree $q\geq 0$. Indeed in this case, we can apply the Leray-Hirsch theorem
for singular cohomology~\cite[Exercise 3 p. 51]{Hatcher:Serrespectral}
(see also~\cite[(17.8.1)]{Dieck:algtop} using Remark~\ref{changement de fibres} or ~\cite[Theorem 4.4]{Mimura-Toda:topliegroups} where the fibre of the fibration is assumed to be path-connected).
In Lemma~\ref{tncz implique classe d'euler nulle cohomologie singuliere} below, we improve further Lemma~\ref{tncz implique classe d'euler nulle cohomologie generalisee} for singular cohomology. The proof of Lemma~\ref{tncz implique classe d'euler nulle cohomologie singuliere} relies on the following interesting Lemma:
\end{rem}
\begin{lem}\label{differentiel dans SSS fibration avec section}
Let $F\buildrel{i}\over\hookrightarrow E\buildrel{p}\over\twoheadrightarrow B$ be
a (Serre) fibration with path-connected base $B$ which admits a section $\sigma:B\rightarrow E$
up to homotopy, i. e. $p\circ \sigma\thickapprox id_B$ and an element $e\in H^m(B)$
such that
$$
\text{Ker }\sigma^*\cup p^*(e)=\{0\}.
$$
Consider the cohomological Serre spectral sequence $(E_r^{*,*},d_r)$ associated
to the fibration $p$.
Suppose that the action of $\pi_1(B)$ on $H^*(F)$
is trivial.

Let $f$ be an element of $\tilde{H}^*(F)$.
Suppose that $f$ is in the image of
$i^*:H^*(E)\rightarrow H^*(F)$.

Denote by $e\otimes 1$, the image of $e$ by the canonical morphism~\cite[III.2.10.10]{Mimura-Toda:topliegroups}
$$
i_*:H^m(B)\rightarrow H^m(B;H^0(F))=E_2^{m,0}
$$

Denote by $1\otimes f$ the image of $f$ by the inverse of the canonical morphism~\cite[III.2.10.9]{Mimura-Toda:topliegroups}
$$
i^*_0:E_2^{0,*}=H^0(B;H^*(F))\buildrel{\cong}\over\rightarrow H^0(*;H^*(F))=H^*(F).
$$
Then the product $e\otimes 1\cup 1\otimes f\in E_2^{m,*}$ must be killed: there exist $r\geq 2$ and $x\in E_r^{*,*}$ such that $d_r(x)=e\otimes 1\cup 1\otimes f$.
\end{lem}
\begin{proof}
Since $B$ is path-connected, the triviality of the local coefficients $H^*(F)$ implies that $i^*_0$ is an isomorphism
(and conversely by~\cite[III.1.18 (3)]{Mimura-Toda:topliegroups}). 

For degree reasons, $\forall r\geq 2, d_r(e\otimes 1)=0$
and $e\otimes 1$ in $E_\infty^{m,0}=F^mH^m(E)$ is $p^*(e)$. 
Since $f\in\text{Im }i^*$, $\forall r\geq 2, d_r(1\otimes f)=0$.
Let $q$ be the degree of $f$.
Let $c\in\tilde{H}^q(E)$ such that $\tilde{H}^*(i)(c)=f$.
As explained in the proof of Lemma~\ref{tncz implique classe d'euler nulle cohomologie generalisee},
$c$ can be chosen in $\text{Ker }\sigma^*$.
Then $1\otimes f$ in $E_\infty^{0,q}=H^q(E)/F^1 H^q(E)$
is the class of $c$. Therefore $e\otimes 1\cup 1\otimes f$ in $E_\infty^{m,q}=F^mH^{m+q}(E)/F^{m+1}H^{m+q}(E)$
is the class of $p^*(e)\cup c$.
Since $
\text{Ker }\sigma^*\cup p^*(e)=\{0\}
$,
$p^*(e)\cup c=0$.
Therefore $e\otimes 1\cup 1\otimes f$ must be killed.
\end{proof}
\begin{lem}\label{tncz implique classe d'euler nulle cohomologie singuliere}
Let $F\buildrel{i}\over\hookrightarrow E\buildrel{p}\over\twoheadrightarrow B$ be
a (Serre) fibration with path-connected base $B$ which admits a section $\sigma:B\rightarrow E$
up to homotopy, i. e. $p\circ \sigma\thickapprox id_B$ and an element $e\in H^m(B)$
such that
$$
\text{Ker }\sigma^*\cup p^*(e)=\{0\}.
$$
Suppose that $\forall n\in\mathbb{N}$, $H_n(B)$ is a finitely generated ${\Bbbk}$-module
or $\forall q\in\mathbb{N}$, $H^q(F)$ is a finitely generated ${\Bbbk}$-module.
Suppose also that $\forall q\geq 2$, $H^q(F)$ is a torsion free ${\Bbbk}$-module.

If the fibration $p$ is Totally Non-Cohomologous to Zero, i. e. $H^*(i):H^*(E)\twoheadrightarrow H^*(F)$ is onto and if $H^m(B)$ is a free ${\Bbbk}$-module then $e=0$ or $\tilde{H}^*(F)= \{0\}$.
\end{lem}
\begin{proof}
Since $H^*(i)$ is onto, the action of
$\pi_1(B)$ on $H^*(F)$ is trivial~\cite[III.Theorem 4.4]{Mimura-Toda:topliegroups}.

By hypothesis, $H^q(F)$ is a finitely generated
${\Bbbk}$-module or $\forall n\geq 0$ $H_n(B)$ is a finitely generated
${\Bbbk}$-module.
So since ${\Bbbk}$ is a principal ideal domain
by~\cite[Theorem 5.5.10]{Spanier:livre}, we have a short exact sequence
$$
0\rightarrow H^p(B)\otimes H^q(F)\buildrel{\mu}\over\rightarrow
E^{p,q}_2\rightarrow\text{Tor}^{\Bbbk}(H^{p+1}(B),H^q(F))\rightarrow 0
$$
where $\mu$ is a morphism of algebras.
Therefore, since $\forall q\geq 0$ $H^q(F)$ is torsion free,
$\forall p,q\in\mathbb{N}$ $E^{p,q}_2\cong H^p(B)\otimes H^q(F)$ as algebras.
Since $i^*$ is onto, $d_r$ is null on $E^{0,q}$.
Therefore the Serre spectral sequence collapses on the $E_2$-term
(Here we have reproved the well-known~\cite[III.Theorem 4.4]{Mimura-Toda:topliegroups}
with weaker hypothesis).

So by Lemma~\ref{differentiel dans SSS fibration avec section},
for any $f\in\tilde{H}^q(F)$, the element $e\otimes f\in E^{m,q}=H^m(B)\otimes H^q(F)$ must be zero.
Since $H^m(F)$ is free, $e\otimes f=0$ implies that $e=0$ or $f$ has torsion.
\end{proof}
The following lemma is a generalization of Lemma~\ref{tncz implique classe d'euler nulle cohomologie singuliere}
if the base $B$ of the fibration is not path-connected.
\begin{lem}\label{tncz implique classe d'euler nulle non connexe}
Let $B=\cup_{\beta\in\pi_0(B)}B_\beta$ the decomposition of $B$ into its path-connected components.
Let $p: E\twoheadrightarrow B$ be
a (Serre) fibration with base $B$ which admits a section $\sigma:B\rightarrow E$
up to homotopy, i. e. $p\circ \sigma\thickapprox id_B$ and an element
$e=(e_\beta)_{\beta\in\pi_0(B)}\in H^m(B)=\Pi_{\beta\in\pi_0(B)}H^m(B_\beta)$
such that
$$
\text{Ker }\sigma^*\cup p^*(e)=\{0\}.
$$
Denote by $F_\beta$, the fibre $p^{-1}(b)$ when $b\in B_\beta$.
For all $\beta\in\pi_0(B)$, suppose that $\forall n\in\mathbb{N}$, $H_n(B_\beta)$ is a finitely generated ${\Bbbk}$-module
or $\forall q\in\mathbb{N}$, $H^q(F_\beta)$ is a finitely generated ${\Bbbk}$-module.
Suppose also that $\forall \beta\in\pi_0(B)$, $\forall q\geq 2$, $H^q(F_\beta)$ is a torsion free ${\Bbbk}$-module.

If the fibration $p$ is Totally Non-Cohomologous to Zero, i. e. $\forall \beta\in\pi_0(B)$ $H^*(i_\beta):H^*(E)\twoheadrightarrow H^*(F_\beta)$ is onto and if $\forall \beta\in\pi_0(B)$ $H^m(B_\beta)$ is a free ${\Bbbk}$-module then $\forall \beta\in\pi_0(B)$ ($e_\beta=0$ or $\tilde{H}^*(F_\beta)= \{0\}$).
\end{lem}
\begin{proof}
For all $\beta\in\pi_0(B)$, we apply Lemma~\ref{tncz implique classe d'euler nulle cohomologie singuliere}
to the fibration 
$$F_\beta\buildrel{i_\beta}\over\hookrightarrow E_{\pi_0(p)^{-1}(\beta)}\buildrel{p_\beta}\over\twoheadrightarrow B_\beta$$
obtained by restricting $p$ to the union $E_{\pi_0(p)^{-1}(\beta)}$ of path-connected components $\alpha$ of $E$ such
that $ \pi_0(p)(\alpha)=\beta$. The fibration $p_\beta$ admits the restriction of $\sigma$ to $B_\beta$,
$\sigma_\beta:B_\beta\rightarrow E_{\pi_0(p)^{-1}(\beta)}$ as section up to homotopy.
The product of maps $$\Pi_{\beta\in\pi_0(B)}\sigma^*_\beta:\Pi_{\beta\in\pi_0(B)}H^*(E_{\pi_0(p)^{-1}(\beta)})\rightarrow \Pi_{\beta\in\pi_0(B)}H^*(B_\beta)$$ can be identified with $\sigma^*:H^*(E)\rightarrow H^*(B)$.
Therefore $\text{Ker }\sigma^*\cup p^*(e)$ can be identified with
$\Pi_{\beta\in\pi_0(B)}(\text{Ker }\sigma^*_\beta\cup p^*_\beta(e_\beta))$.
\end{proof}
Of course, the following theorem generalizes the fundamental Corollary 11.4 of~\cite{Milnor-Stasheff}.
\begin{theor}\label{pull-back tncz implique cohomologie fibre ou classe d'euler nulle}
Let $g:G\hookrightarrow E$ be the pull-back of an embedding in the sense of definition~\ref{pull-back fibre ou transverse d'un embedding}.
Let $e_{\nu}\in h^m(M)$ be the Euler class of the normal bundle of the embedding $\phi:M\hookrightarrow B$.
Let $p_g:E^I\times_g G\twoheadrightarrow E$ be the fibration associated to $g$ defined by
$
p_g((\omega, b))=\omega(0)
$ for any $b\in G$ and any path $\omega:I\rightarrow E$ such that $\omega(1)=g(b)$.
Let $p_\phi:B^I\times_\phi M\twoheadrightarrow B$ be the fibration associated to $\phi$.

Suppose that $\forall n\in\mathbb{N}$, $H_n(G)$ is a finitely generated ${\Bbbk}$-module
or $\forall y\in B$, $\forall q\in\mathbb{N}$, $H^q(p_\phi^{-1}(y))$ is a finitely generated ${\Bbbk}$-module.
Suppose also that $\forall y\in B$, $\forall q\geq 2$, $H^q(p_\phi^{-1}(y))$ is a torsion free ${\Bbbk}$-module.

If the fibration $p_g$ is Totally Non-Cohomologous to Zero, i. e.
$\forall x\in E$ $H^*(i_x):H^*(E^I\times_g G)\twoheadrightarrow H^*(p_g^{-1}(x))$ is onto and if  $H^m(G)$ is a free ${\Bbbk}$-module then $\forall b\in G$ either $\tilde{H}^*(p^{-1}_\phi(p\circ g(b)))= \{0\}$ or the component of
$q^*(e_\nu)$ in $H^m(G_{[b]})$ the cohomology of the path-connected component of $b$ in $G$ is trivial.
\end{theor}
\begin{proof}
By definition of the homotopy fibre product $\leftidx{^f}{E}{^g}$, for any $a\in F$, we have the following commutative diagram
of spaces
$$
\xymatrix{
ev_0^{-1}(a)\ar[r]^-\cong\ar[d]
&p_g^{-1}(f(a))\ar[d]^{i_{f(a)}}\ar[r]^-\approx
&p_\phi^{-1}(p\circ f(a))\ar[d]\\
\leftidx{^f}{E}{^g}\ar[r]\ar[d]_{ev_0}
& E^I\times_g G\ar[d]_{p_g}\ar[r]^{p^I\times q}
& B^I\times_\phi M\ar[d]_{p_\phi}
& M\ar[ld]^{\phi}\ar[l]_-s^-\approx\\
F\ar[r]_f
&E\ar[r]_p
&B}
$$
where the bottom left square is a pull-back, $s$ is a homotopy equivalence and the bottom right square is a homotopy pull-back
(i. e. the induced map $p_g^{-1}(f(a))\rightarrow p_\phi^{-1}(p\circ f(a))$ between the homotopy fibre of $g$ and $\phi$ is a homotopy equivalence).
Since $\forall a\in F$, $H^*(i_{f(a)}):H^*(E^I\times_g G)\rightarrow H^*(p_g^{-1}(f(a)))$ is onto,
$H^*(\leftidx{^f}{E}{^g})\rightarrow H^*(ev_0^{-1}(a))$ is also onto, i. e. $ev_0$ is Totally Non-Cohomologous to Zero.

Suppose now that $f=g$ (and $F=G$). By part 4) of Theorem~\ref{formule open string coproduct},
$$
\text{Ker } \sigma^*\cup ev_0^*\circ q^*(e_{\nu})=\{0\}.
$$
By applying Lemma~\ref{tncz implique classe d'euler nulle non connexe} to the fibration
$ev_0:\leftidx{^g}{E}{^g}\twoheadrightarrow G$ we obtain that for all $b\in G$, the component of $q^*(e_\nu)$ in
$H^*(G_{[b]})$ is trivial or $\tilde{H}^*(ev_0^{-1}(b))= \{0\}$.
\end{proof}
\section{an example}
\begin{cor}\label{pull-back fibration et espace projectif}
Consider the following pull-back diagram
$$
\xymatrix{
G\ar[r]^g\ar[d]_q
&E\ar[d]^p\\
\mathbb{CP}^q\ar[r]_\phi
&\mathbb{CP}^n
}
$$
where $p:E\twoheadrightarrow \mathbb{CP}^n$ is a (Serre) fibration over the $n$-th complex projective space $ \mathbb{CP}^n$ and $\phi: \mathbb{CP}^q\hookrightarrow \mathbb{CP}^n$ is the inclusion, $0\leq q<n$.

If the fibration $p_g$ associated to $g$ is Totally Non-Cohomologous to Zero and if  $H^{2n-2q}(G)$ is a free ${\Bbbk}$-module then $q^*(a^{n-q})=0$. Here $a$ is a generator of $H^2(\mathbb{CP}^q)$.
\end{cor}
\begin{proof}
By~\cite[Theorem 14.10]{Milnor-Stasheff}, $c(T\mathbb{CP}^n)$, the total Chern class of the tangent
bundle of $\mathbb{CP}^n$ is equal to $(1+a)^{n+1}$ in $H^{\Pi}(\mathbb{CP}^n)$.
Since $T\mathbb{CP}^q\oplus \nu=T\mathbb{CP}^n_{|\mathbb{CP}^q}$, in $H^{\Pi}(\mathbb{CP}^q)$,
$$
c(\nu)=c(T\mathbb{CP}^n_{|\mathbb{CP}^q})/ c(T\mathbb{CP}^q)=(1+a)^{n+1}/(1+a)^{q+1}=(1+a)^{n-q}.
$$
Therefore $e(\nu)=c_{n-q}(\nu)=a^{n-q}$.

We have a morphism of $S^1$-principal fibre bundles
$$
\xymatrix{
S^{2q+1}\ar[d]\ar[r]^{\tilde{\phi}}
& S^{2n+1}\ar[d]\\
\mathbb{CP}^q\ar[r]^{\phi}
&\mathbb{CP}^n
}
$$
where $\tilde{\phi}:S^{2q+1}\hookrightarrow S^{2n+1}$ is the inclusion.
Since this square is a pull-back, the homotopy fibre of $\phi$, $p_{\phi}^{-1}(*)$, is homotopy equivalent to the homotopy fibre of $\tilde{\phi}$.
Since $\pi_{2q+1}(S^{2n+1})=\{0\}$, $\tilde{\phi}$ is homotopically trivial and its homotopy fibre
is homotopy equivalent to $S^{2q+1}\times\Omega S^{2n+1}$.
Therefore by Theorem~\ref{pull-back tncz implique cohomologie fibre ou classe d'euler nulle},
all the components of $q^*(e_\nu)$ are trivial. 
\end{proof}
\begin{cor}\label{pull back d'inclusions espaces projectif}
Let $f: \mathbb{CP}^p\hookrightarrow \mathbb{CP}^n$ 
and $g: \mathbb{CP}^q\hookrightarrow \mathbb{CP}^n$ be the inclusions, $0\leq p<n$, $0\leq q<n$.
If the fibration $ev_0:\leftidx{^f}{E}{^g}\twoheadrightarrow \mathbb{CP}^p$ is Totally Non-Cohomologous to Zero then $q\geq p$ and $n>p+q$.
\end{cor}
\begin{proof}
Suppose that $q<p$. Since $ev_0^*\circ f^*=ev_1^*\circ g^*$ and $f^*$ is surjective, for all 
$q<i\leq p$, $$ev_0^*(a^i)=ev_0^*\circ f^*(a^i)=ev_1^*\circ g^*(a^i)=ev_1^*(a^i)=0.$$
Therefore $ev_0^*$ is not injective and so by~\cite[III.Theorem 4.4]{Mimura-Toda:topliegroups},
$ev_0$ is not Totally Non-Cohomologous to Zero.

Consider the Serre spectral sequence associated to the fibration
$F\hookrightarrow \leftidx{^f}{E}{^g}\buildrel{ev_1}\over\twoheadrightarrow \mathbb{CP}^q$.
We saw in the proof of Corollary~\ref{pull-back fibration et espace projectif} that its fibre
$F$ is homotopy equivalent to $S^{2p+1}\times \Omega S^{2n+1}$. Therefore $H^+(F)$ is concentrated
in degree $\geq 2p+1$. And so $E_r^{s,t}\neq \{0\}\Rightarrow t=0$ or $t\geq 2p+1$.
Therefore $ev_1^*$ is an isomorphism in degree $\leq 2p$.
(In particular, if $p\geq q$, $ev_1^*$ is injective).

Suppose now that $ev_0$ is Totally Non-Cohomologous to Zero and that $n-q\leq p$.
Then $H^{2n-2q}(ev_1):H^{2n-2q}(\mathbb{CP}^q)\buildrel{\cong}\over\rightarrow H^{2n-2q}(\leftidx{^f}{E}{^g})$ is an isomorphism. Since $H^{2n-2q}(\mathbb{CP}^q)$ is $\Bbbk$-free,
by Corollary~\ref{pull-back fibration et espace projectif}, $H^{2n-2q}(ev_1)(a^{n-q})=0$.
And so $a^{n-q}=0$ in $H^*(\mathbb{CP}^q)$. Therefore $n-q>q$. In particular, $p>q$.
\end{proof}
\begin{rem}
In the case $p=q$ of Corollary~\ref{pull back d'inclusions espaces projectif}, parts 4) and 8) of
Theorem~\ref{formule open string coproduct} give that
$$
\text{Ker } \sigma^*\cup ev_0^*(a^{n-q})=\{0\}
$$
and that for all $b\in H^*(\leftidx{^g}{E}{^g})$ of degree $>4q-2n$, 
$
b\cup ev_0^*(a^{n-q})=0.
$
\end{rem}
\begin{rem} (Over $\mathbb{Q}$, the converse of Corollary~\ref{pull back d'inclusions espaces projectif} is true)
Over $\mathbb{Q}$, a relative Sullivan model of $ev_0$ is given by the inclusion of differential graded algebras
$$
(\Lambda (x_2,z_{2p+1}),d)\hookrightarrow (\Lambda (x_2,z_{2p+1},t_{2q+1},sy_{2n+1}),d)
$$
with $d(z_{2p+1})=x_2^{p+1}$, $d(t_{2q+1})=x_2^{q+1}$ and  $d(sy_{2n+1})=t_{2q+1}x_2^{n-q}-z_{2p+1}x_2^{n-p}$(Compare with~\cite[Example 7.3]{MenichiL:cohaf}).

If $n>p+q$, by replacing $sy$ by $sy-ztx^{n-p-q-1}$, we can assume that $d(sy)=0$.
If $q\geq p$, by replacing $t$ by $t-zx^{q-p}$, we can assume that $d(t)=0$.
Therefore if $n>p+q$ and $q\geq p$ then over $\mathbb{Q}$, $ev_0$ is Totally Non-Cohomologous to Zero.
\end{rem}
\part{the free loops case}
In this part, we consider our main example of homotopy fibre product, the space $LM$ of free loops on a manifold.
\section{The loop product and the loop coproduct}
Let $M$ be a smooth oriented manifold without boundary.
In this section, $M$ is not necessarily compact.
The diagonal map $\Delta:M\hookrightarrow M\times M$ is an embedding.
Since $M$ is Hausdorff, $\Delta(M)$ is a closed subset of
$M\times M$.
As we have explained
in Section~\ref{shriek d'un embedding}, we can define the shriek map
of $\Delta$, $\Delta_!$ in homology.

By definition, the {\it intersection product} in homology, is the composite
$$
H_*(M)\otimes H_*(M)\buildrel{\times}\over\rightarrow H_*(M\times M)
\buildrel{\Delta_!}\over\rightarrow  H_{*-m}(M).
$$

We have the following push-out squares
$$
\xymatrix{
S^1\coprod S^1\ar[r]
&S^1\vee S^1
&S^1\ar[l]_{c}\\
\star\coprod\star\ar[u]\ar[r]
&\star\ar[u]
&S^0\ar[u]\ar[l]
}
$$
where $c:S^1\rightarrow S^1\vee S^1$ is the comultiplication or pinch map of $S^1$.
Note that all the vertical maps are cofibrations.
Since the functor $map(-,M)$ transforms push-out squares in pull-back squares,
we have the following pull-back squares where all the vertical maps are fibrations
$$
\xymatrix{
LM\times LM\ar[d]_{ev\times ev}
&LM\times_M LM\ar[d]_q\ar[l]_{\tilde{\Delta}}\ar[r]^\mu
&LM\ar[d]^{(ev,ev_{1/2})}\\
M\times M
&M\ar[r]_\Delta\ar[l]^\Delta
&M\times M
}
$$
and $\mu:=map(c,M)$ is the composition or multiplication of loops.
Since $\Delta:M\hookrightarrow M\times M$ is an embedding, as we have explained
in Section~\ref{shriek d'un pull-back d'un embedding}, we can define the shriek map
of $\tilde{\Delta}$, $\tilde{\Delta}_!$ in homology, and the shriek maps
of $\mu$, $\mu_!$ in homology, $\mu^!$ in cohomology.

By definition, the Chas-Sullivan loop product in homology, is the composite
$$
H_*(LM)\otimes H_*(LM)\buildrel{\times}\over\rightarrow H_*(LM\times LM)
\buildrel{\tilde{\Delta}_!}\over\rightarrow H_{*-m}(LM\times_M LM)
\buildrel{\mu_*}\over\rightarrow H_{*-m}(LM).
$$
By definition, the loop coproduct in homology is the composite
$$
H_*(LM)\buildrel{\mu_!}\over\rightarrow H_{*-m}(LM\times_M LM)
\buildrel{\tilde{\Delta}_*}\over\rightarrow H_{*-m}(LM\times LM).
$$
In this note, we work over an arbitrary principal ideal domain $\Bbbk$
and so the cross product is not in general an isomorphism.
Therefore, we will consider the loop coproduct in cohomology.
By definition, the loop coproduct in cohomology is the product defined
by the composite
$$
H^*(LM)\otimes H^*(LM)\buildrel{\times}\over\rightarrow H^*(LM\times LM)
\buildrel{\tilde{\Delta}^*}\over\rightarrow H^{*}(LM\times_M LM)
\buildrel{\mu^!}\over\rightarrow H^{*+m}(LM)
$$

\begin{rem}
Let $k:\Omega M\hookrightarrow LM$ be the inclusion of the pointed loops into the free loops.
If the dimension of $M$ is positive, from Corollary~\ref{shriek et inclusion fibre nulle},
we obtain that the composite
$$
H_*(\Omega M)\otimes H_*(\Omega M)
\buildrel{k_*\otimes k_*}\over\rightarrow
H_*(LM)\otimes H_*(LM)
\buildrel{\text{loop product}}\over\rightarrow
H_{*-m}(LM)
$$
is trivial.
\end{rem}
\section{A simple formula for the loop coproduct}
Denote by $LM_{[1]}$ the path-connected component of $LM$ of freely contractile loops.
Recall that $ev:LM\twoheadrightarrow M$ is the evaluation map.
Let $\sigma:M\hookrightarrow LM$, $m\mapsto \text{constant loop }m$, be its trivial section.

\begin{theor}\label{formule loop coproduct}
Let $M$ be a connected, closed  $\Bbbk$-oriented manifold of dimension $m$.
Let $\omega\in H^m(M)$ be its orientation class.
Let $\chi(M)$ be its Euler characteristic.
Then 

\noindent 1) The loop coproduct, $\mu^!\circ\tilde{\Delta}^*$ on $H^*(LM)$ is
given for $a$, $b\in H^*(LM)$, by
\begin{equation*}
\mu^!\circ\tilde{\Delta}^*(a\otimes b)
=\chi(M) a\cup ev^*(\sigma^*(b)\cup \omega).
\end{equation*}
Here $\cup$ is the cup product on $H^*(LM)$.

\noindent 2) The loop coproduct, $\mu^!\circ\tilde{\Delta}^*$ on $H^*(LM)$ is graded commutative
with respect to the usual degrees:
that is, for $a\in H^p(LM)$, $b\in H^q(LM)$

$\mu^!\circ\tilde{\Delta}^*(a\otimes b)
=(-1)^{pq}\mu^!\circ\tilde{\Delta}^*(b\otimes a).
$

\noindent 3) The ideal $\text{Ker } \sigma^*:H^*(LM)\twoheadrightarrow H^*(M)$ satisfies
$$
\text{Ker } \sigma^*\cup   \chi(M) ev^*(\omega)=\{0\}.
$$

\noindent 4) The loop coproduct, $\mu^!\circ\tilde{\Delta}^*$ on $H^*(LM)$ is
given for $a$, $b\in H^*(LM)$, by
\begin{equation*}
\mu^!\circ\tilde{\Delta}^*(a\otimes b)
=\chi(M) ev^*(\sigma^*(a)\cup \sigma^*(b)\cup \omega).
\end{equation*}

\noindent 5) the loop coproduct, $\mu^!\circ\tilde{\Delta}^*$ on $H^*(LM)$ is
trivial outside of $H^0(LM_{[1]})\otimes H^0(LM_{[1]})\cong \Bbbk\otimes \Bbbk$.

\noindent 6) On $H^0(LM_{[1]})\otimes H^0(LM_{[1]})$, the loop coproduct is given by
$$\mu^!\circ\tilde{\Delta}^*(1\otimes 1)=\chi(M)ev^*(\omega).$$

\noindent 7) The image of the loop coproduct  $\mu^!\circ\tilde{\Delta}^*$ is contained
in $H^*(LM_{[1]})$.
\end{theor}
\begin{rem}
Over a field, parts 2), 5) and 7) of this Theorem are not new.
Indeed over a field, the commutativity of the loop coproduct
was proved by Cohen and Godin~\cite{1095.55006} and parts 5) and 7) are the
duals of ~\cite[Theorem B (2)]{tamanoi-2007}.
\end{rem}
\begin{lem}\label{open string coproduct redonne loop coproduct} 
Consider $\leftidx{^\Delta}{M}{^\Delta}$, the self homotopy fibre product along the diagonal. 
Explicitly $\leftidx{^\Delta}{M}{^\Delta}$ is just the subspace
 $$\{(\omega,\omega')\in M^I\times M^I/ \omega(0)=\omega'(0),\;\omega(1)=\omega'(1)\}.$$
 Let $\Theta:\leftidx{^\Delta}{M}{^\Delta}\buildrel{\cong}\over\rightarrow LM$ be the homeomorphism
 mapping $(\omega,\omega')$ to the free loop $\omega *\omega'^{-1}$ obtained by composing the path
 $\omega$ with the inverse of the path $\omega'$. Then
 
 1) ~\cite[Example iii) free loop space]{Sullivan:openclosedstring} With respect to the loop product
 and the open string product,
 $$H_*(\Theta):H_*(\leftidx{^\Delta}{M}{^\Delta})\buildrel{\cong}\over\rightarrow H_*(LM)$$
 is an isomorphism of algebras.
 
 2) With respect to the loop coproduct
 and the open string coproduct,
 $$H^*(\Theta):H^*(LM)\buildrel{\cong}\over\rightarrow H^*(\leftidx{^\Delta}{M}{^\Delta})$$
 is an isomorphism of algebras.
\end{lem}
\begin{proof}
Denote by $\rho_\alpha:LM\buildrel{\cong}\over\rightarrow LM$ the homeomorphism mapping a free loop $l$
to the rotated free loop $t\mapsto l(t+\alpha)$. 
Up to the homeomorphism $\Theta$, the two pull-back squares defining the open string (co)product on
$$
\xymatrix{
\leftidx{^\Delta}{M}{^\Delta}\times\leftidx{^\Delta}{M}{^\Delta}\ar[d]_{ev_1\times ev_0}
&\leftidx{^\Delta}{M}{^\Delta}\times_{M\times M} \leftidx{^\Delta}{M}{^\Delta}\ar[l]_-{\tilde{\Delta}}\ar[r]^-{\mu_{\Delta,\Delta,\Delta}}\ar[d]_{ev_{1/2}}
& \leftidx{^\Delta}{M}{^\Delta}\ar[d]^{ev_{1/2}}\\
M\times M
& M\ar[l]^\Delta\ar[r]_{\Delta}
&M\times M
}
$$
coincide with the following two vertical rectangles defining the loop (co)product since $\rho_\alpha$ is homotopic to the identity map.
$$
\xymatrix{
LM\times LM\ar[d]_{\rho_{1/2}\times id}^\cong
&LM_{1/2}\times_M LM_0\ar[d]_\cong\ar[l]\ar[r]
&LM\ar[d]^{\rho_{1/4}}_\cong\\
LM\times LM\ar[d]_{ev\times ev}
&LM\times_M LM\ar[d]_q\ar[l]_{\tilde{\Delta}}\ar[r]^\mu
&LM\ar[d]^{(ev,ev_{1/2})}\\
M\times M
&M\ar[r]_\Delta\ar[l]^\Delta
&M\times M
}
$$
\end{proof}
\begin{proof}[Proof of Theorem~\ref{formule loop coproduct}]
We apply Theorem~\ref{formule open string coproduct} in the case where $f=g=h=\phi$ is the diagonal embedding
$\Delta:M\hookrightarrow M\times M$.

The normal bundle $\nu$ of $\Delta$ is isomorphic
to the tangent bundle of $M$, $TM$~\cite[Lemma 11.5]{Milnor-Stasheff}.
Since $M$ is compact and connected,  the Euler class of the tangent bundle is the fundamental
class multiplied by the Euler characteristics~\cite[Corollary 11.12]{Milnor-Stasheff}:
$$
e_\nu=e_{TM}=\chi(M)\omega.
$$
Using part 2) of Lemma~\ref{open string coproduct redonne loop coproduct}, we have proved Theorem~\ref{formule loop coproduct}.
\end{proof}
\begin{rem}\label{loop coproduit noncompact}
Let $M$ be a connected, non-compact  $\Bbbk$-oriented manifold of dimension $m$
and suppose that  $\Bbbk$ is a field.
Then its loop coproduct 
$\mu^!\circ\tilde{\Delta}^*$ on $H^*(LM)$ is trivial.
\end{rem}
\begin{proof}[Proof of Remark~\ref{loop coproduit noncompact}]
Since $M$ is non-compact then $H_m(M)=0$.
Since $\Bbbk$ is a field, $H^m(M)=\text {Hom}_\Bbbk(H_m(M),{\Bbbk})=0$.
So $e_{TM}$ is trivial.
Therefore the same proof as the proof of Theorem~\ref{formule loop coproduct}
shows that the loop coproduct is trivial.

Alternatively, for any $x\in H^*(LM\times_M LM)$
$$
\mu^*\circ\mu^!(x)=x\cup q^*(e_{TM})=0.
$$
Since  the composition of loops $\mu$ admits a section,
$\mu^*$ is injective and so $\mu^!$ is null.
\end{proof}
\begin{cor} 
The loop coproduct is trivial if and only if $\chi(M)=0$ in ${\Bbbk}$.
\end{cor}
This corollary follows also from~\cite[(3-1) and (3-2)]{tamanoi-2007}
(Compare also with~\cite[Corollary 3.2]{tamanoi-2007} or~\cite[Bottom p. 7]{Sullivan:openclosedstring}).
In~\cite{Chataur-Thomas:Frobratloopalg}, Chataur and Thomas gave the first example of manifold with non-trivial
loop coproduct.
\begin{proof}
Since $ev\circ \sigma=id$, $ev^*$ is injective.
Therefore since $w$ is a basis of $H^m(M)$, $$\chi(M)ev^*(\omega)=0 \Longleftrightarrow \chi(M)\omega =0
\Longleftrightarrow \chi(M)=0\text{ in }\Bbbk
.$$
\end{proof}
\begin{rem}\label{noyau section egale cohomologie relative lacets libres lacets constants}
Since $ev\circ \sigma=id_M$, $M$ is a subspace of $LM$ and we can consider the relative cohomology $H^*(LM,M)$. Using the long exact sequence associated,
$H^*(LM,M)$ can be identified with
$\text{Ker } \sigma^*:H^*(LM)\twoheadrightarrow H^*(M)$.
From part 3) of Theorem~\ref{formule loop coproduct}, we have that the loop coproduct vanishes on $H^*(LM,M)$.
In~\cite{Sullivan:openclosedstring}, Sullivan introduced a non trivial product
on $H^*(LM,M)$ of different degree that he called the cutting at any time $\vee$.
In~\cite{Goresky-Hingston:loopclosedgeodesics},
Goresky and Hingston rediscover this non trivial product that they denote $\circledast$.
\end{rem}
\section{Applications}
\begin{theor}\label{Applications}
Let $M$ be a connected, closed  $\Bbbk$-oriented manifold of dimension $m$.
Let $\omega\in H^m(M)$ be its orientation class.
Let $\chi(M)$ be its Euler characteristic.
Then 

1) $\chi(M) ev^*(\omega)\in H^m(LM_{[1]})$.

2) For any $a\in H^*(LM)$ of positive degree, 
$$\chi(M)a\cup ev^*(\omega)=0.$$
\end{theor}
\begin{proof}
Comparing 6) and 7) in Theorem~\ref{formule loop coproduct}, we get 1).

By 5) and 1) in Theorem~\ref{formule loop coproduct},
$$
0=\mu^!\circ\tilde{\Delta}^*(a\otimes 1)
=\chi(M) a\cup ev^*(\omega).
$$
\end{proof}
If  $\Bbbk$ is a field
then 1) means that for all non contractile free loop $\alpha$
and for all $a\in H_m(LM_{[\alpha]})$, $$\chi(M)H_m(ev)(a)=0.$$
\begin{rem}
In general, $ev^*(\omega)$ does not belong to $H^*(LM_{[1]})$
and $a\cup ev^*(\omega)$ is not trivial:
Suppose that $\Bbbk$ is a field.
Let $G$ be a connected compact Lie group. Note that $\chi(G)=0$.
For any $[\alpha]\in\pi_1(G)$, let $\Theta_\alpha$ be the usual isomorphism
from the tensor product   $H_*(\Omega_{[\alpha]} G)\otimes H_*(G)$ to   $H_*(LG_{[\alpha]})$. Here $\Omega_{[\alpha]} G$ denotes the pointed loops of $G$ homotopic
to $\alpha$. Let $\varepsilon$ be the augmentation of $H_*(\Omega_{[\alpha]} G)$.
The previous isomorphism $\Theta_\alpha$ fits into the commutative triangle
of graded vector spaces
$$
\xymatrix{
H_*(\Omega_{[\alpha]} G)\otimes H_*(G)
\ar[rr]^{\Theta_\alpha}_\cong\ar[rd]_{\varepsilon\otimes Id}
&& H_*(LG_{[\alpha]})\ar[ld]^{H_*(ev)}\\
&\Bbbk\otimes  H_*(G)
}
$$
Let $[G]$ be the fundamental class of $G$.
Recall that $[\alpha]$ is a generator of $H_0(\Omega_{[\alpha]} G)$.
Then
$$
H_{\text{dim }G}(ev)\circ\Theta_\alpha([\alpha]\otimes [G])=[G]\neq 0
$$
Therefore  $ev^*(\omega)$ does not belong to $H^{\text{dim }G}(LG_{[1]})$
for any non simply-connected, connected compact Lie group $G$ (e. g. $S^1$).

Let $\eta:{\Bbbk}\rightarrow H^*(\Omega G)$ be the unit map of $H^*(\Omega G)$.
The usual isomorphism of algebras $\Theta$ from the tensor product
of graded algebras $H^*(\Omega G)\otimes H^*(G)$ to   $H^*(LG)$
fits similarly into the commutative triangle
of graded algebras
$$
\xymatrix{
H^*(\Omega G)\otimes H^*(G)
\ar[rr]^{\Theta}_\cong
&& H^*(LG)\\
&\Bbbk\otimes  H^*(G)\ar[ul]^{\eta\otimes Id}\ar[ur]_{H^*(ev)}
}
$$
Therefore for any non-zero element $b$ of $H^*(\Omega G)$,
$$
\Theta(b\otimes 1)\cup H^*(ev)(\omega)=\Theta(b\otimes w)\neq 0.
$$
If $G$ is a connected compact Lie group such that $H^*(\Omega G)$
is not concentrated in degree $0$ (e. g. $S^3$), we have obtained an element
$a$ of positive degree such that $a\cup ev^*(\omega)$ is non zero.
\end{rem}
\begin{cor}\label{point fixe action du cercle a homotopie pres}
Let $M$ be a connected, closed  $\Bbbk$-oriented manifold of dimension $m$
such that in ${\Bbbk}$, $\chi(M)\neq 0$.
Let $\mu:S^1\times M\rightarrow M$ be a continuous map such that
the composite
$$
\{1\}\times M\rightarrow S^1\times M\buildrel{\mu}\over\rightarrow M
$$
is homotopic to the identity map.
Then there exists an map $\nu:S^1\times M\rightarrow M$ homotopic to $\mu$
who has at least a fixed point $m_0$, i. e. $\nu(S^1\times\{m_0\})=m_0$.
\end{cor}
\begin{proof}
Let $\sigma_\mu:M\rightarrow LM$ be the map sending $m\in M$ to its orbit
$\mu(-,m):S^1\rightarrow M$.
Since for all $m\in M$, $ev\circ\sigma_\mu(m)=\mu(1,m)$, $\sigma_\mu$
is a section up to homotopy of $ev$.
Therefore in cohomology, $$\sigma_\mu^*\circ ev^*(\chi(M)\omega)=\chi(M)\omega.$$

Since $M$ is path-connected, $\sigma_\mu$ arrives in the path-connected component $LM_{[\alpha]}$ of a free loop $\alpha$. So $\sigma_\mu^*$ is trivial outside
of $H^*(LM_{[\alpha]})$. By 1) of Theorem~\ref{Applications},
$\chi(M) ev^*(\omega)\in H^m(LM_{[1]})$.
Since $\chi(M)$ is not zero in ${\Bbbk}$,
$\chi(M)\omega=\sigma_\mu^*\circ ev^*(\chi(M)\omega)$ is not trivial.
Therefore $\alpha$ is contractile, i. e. $[\alpha]=[1]$.

Let $i:\{m_0\}\hookrightarrow M$ be the inclusion of a non-degenerated base point into $M$. Since $\alpha$ is contractile, $\sigma_\mu(m_0)$ is homotopic to the constant loop $\hat{m_0}$ and so the following triangle commutes up to homotopy.
$$
\xymatrix{
S^1\times M\ar[r]^\mu
&M\\
S^1\times \{m_0\}\ar[u]^{S^1\times i}\ar[ur]_{\hat{m_0}}
}
$$
By the homotopy extension property of the cofibration
$S^1\times i:S^1\times \{m_0\}\hookrightarrow S^1\times M$, we can change up to homotopy $\mu$ into a map $\nu$ such that the triangle commutes now exactly.
\end{proof}
Corollary~\ref{point fixe action du cercle a homotopie pres}
should be considered as an homotopy version of the following classic result:
\begin{theor}(\cite[Theorem 4.7.12]{Spanier:livre}.
Compare also with~\cite[Theorem 5.39 or Corollary 6.17]{Kawakubo:transformationgroups})
Let $M$ be a compact Euclidean Neighborhood Retract (e. g. a compact topological
manifold~\cite[A.9]{Hatcher:algtop}) such that $\chi(M)\neq 0$.
Let $\mu:S^1\times M\rightarrow M$ be an action of the circle on $M$.
Then $M$ has at least a fixed point.
\end{theor}
\begin{rem}
If a map $\mu:S^1\times M\rightarrow M$ is only an action up to homotopy
then it may happen that $M$ has no fixed point.
Therefore the conclusion of
Corollary~\ref{point fixe action du cercle a homotopie pres}
 cannot be improved in general:
Consider the sphere $M=S^2$ in $\mathbb{R}^3$.
Let $\nu:S^1\times S^2\rightarrow S^2$ be the action given by rotation of axis $z$. Let $f:S^2\rightarrow S^2$ be the rotation of angle $\pi$ and of axis $y$.
Since $f$ is homotopic to the identity map, the composite
$f\circ\nu$ is an action up to homotopy without any fixed point.
\end{rem}
\begin{cor}\label{differentiel dans SSS en cohomologie}
Let $M$ be a connected, closed  $\Bbbk$-oriented manifold of dimension $m$.
Consider the cohomological Serre spectral sequence $(E_r^{*,*},d_r)$ associated
to the free loop fibration
$\Omega M\buildrel{i}\over\hookrightarrow LM\buildrel{ev}\over\twoheadrightarrow M$.
Suppose that the (conjugation) action of $\pi_1(M)$ on $H^*(\Omega M)$
is trivial.

Let $f$ be an element of $\tilde{H}^*(\Omega M)$.
Suppose that $f$ is in the image of
$i^*:H^*(LM)\rightarrow H^*(\Omega M)$.
Then $\chi(M)\omega\otimes f\in H^m(M)\otimes H^*(\Omega M)=E_2^{m,*}$ must be killed: there exist $r\geq 2$ and $x\in E_r^{*,*}$ such that $d_r(x)=\chi(M)\omega\otimes f$.
\end{cor}
\begin{proof}
Since $M$ is compact, $\forall n\geq 0$, $H_n(M)$ is a finitely generated
${\Bbbk}$-module.
So since ${\Bbbk}$ is a principal ideal domain
by~\cite[Theorem 5.5.10]{Spanier:livre}, we have a short exact sequence
$$
0\rightarrow H^p(M)\otimes H^q(\Omega M)\buildrel{\mu}\over\rightarrow
E^{p,q}_2\rightarrow\text{Tor}^{\Bbbk}(H^{p+1}(M),H^q(\Omega M))\rightarrow 0
$$
where $\mu$ is a morphism of algebras.
Since $H^0(\Omega M)$,
$H^1(M)\cong\text{Hom}(H_1(M),{\Bbbk})$
and $H^{m+1}(M)=0$ are torsion free, $E^{p,q}_2\cong H^p(M)\otimes H^q(\Omega M)$ if $p=0$ or $p=m$ or $q=0$.

By part 3) of Theorem~\ref{formule loop coproduct},
$
\text{Ker } \sigma^*\cup \chi(M)ev^*(\omega)=\{0\}
$.
So by Lemma~\ref{differentiel dans SSS fibration avec section},
$e\otimes 1\cup 1\otimes f=\chi(M)\omega\otimes f$ must be killed.
\end{proof}
\begin{cor}\label{tncz implique euler nulle}
Let $M$ be a connected, closed  $\Bbbk$-oriented manifold.
Suppose that the free loop fibration
$\Omega M\buildrel{i}\over\hookrightarrow LM\buildrel{ev}\over\twoheadrightarrow M$ is Totally Non-Cohomologous to Zero, i. e. $H^*(i)$ is onto
and that $H^k(\Omega M)$ is a torsion free ${\Bbbk}$-module
for each $k\geq 1$.
Then $\chi(M)=0$ in ${\Bbbk}$ or $M$ is a point.
\end{cor}
\begin{proof}

Suppose that $\chi(M)$ is not equal to zero in ${\Bbbk}$.
By part 3) of Theorem~\ref{formule loop coproduct},
$
\text{Ker } \sigma^*\cup \chi(M)ev^*(\omega)=\{0\}
$. So by Lemma~\ref{tncz implique classe d'euler nulle cohomologie singuliere},
$\tilde{H}^{*}(\Omega M)=\{0\}$. This means that
 $H^{>0}(\Omega M)=\{0\}$ and that $\pi_1(M)=\{0\}$.
So $H^*(M)\cong {\Bbbk}$. Since $H^{\text{dim }M}(M)={\Bbbk}\omega$,
$M$ must be of dimension $0$, i. e. $M$ is a point.
\end{proof}
\noindent{\bf  Interpretation and proofs of Corollaries~\ref{differentiel dans SSS en cohomologie} and~\ref{tncz implique euler nulle} in term of integration along the
basis.}
Let $F\buildrel{i}\over\hookrightarrow E\buildrel{p}\over\twoheadrightarrow M$ be a fibration.
Suppose that $\pi_1(M)$ acts trivially on $H^q(F)$.
Let $\int i:H^q(F)\rightarrow H^{q+m}(E)$
be the composite
$$
H^q(F)\build\rightarrow_\cong^{\omega\otimes Id} H^m(M)\otimes H^q(F)=E_2^{m,q}
\twoheadrightarrow E_\infty^{m,q}=F^mH^{m+q}(E)\subset
H^{m+q}(E).
$$
If $M$ is a sphere, this {\it integration along the basis} $\int i$
appears in Wang exact sequence and the following two properties are well
known~\cite[Theorems (3.1) and (3.5) Chapter VII]{Whitehead:eltsoht}:
for $x\in H^*(F)$ and $y\in H^*(E)$,
$$
(\int i)(i^*(y)\cup x)=y\cup (\int i)(x)
$$
and
$
p^*(\omega)=(\int i)(1)
$.
In general, these properties are easy to deduce from the multiplicativity
and the naturality of the Serre spectral sequence.
In particular, we have $
(\int i)\circ i^*(y)=y\cup p^*(\omega).
$
Since the inclusion of the fibre $i:F\hookrightarrow E$
is the pull-back along $p$ of the embedding $*\hookrightarrow M$,
following Section~\ref{shriek d'un pull-back d'un embedding},
one can define a shriek map $i^!$ for $i$.
In this paper, we will not use that
$i^!$ coincides with $(\int i)$ although this should follows from the diagram in~\cite[(2) p. 12]{leborgne:string}.

Consider the free loop fibration $\Omega M\buildrel{i}\over\hookrightarrow LM\buildrel{ev}\over\twoheadrightarrow M$.
In this case, $i_!$ is the intersection morphism $H_{*+\text{dim }M}(LM)\rightarrow H_*(\Omega M)$ defined by Chas and Sullivan.

Let $f$ be an element as in Corollary~\ref{differentiel dans SSS en cohomologie}. By Theorem~\ref{Applications},
$$
\chi(M)\int i(f)=\chi(M)(\int i)\circ i^*(c)=\chi(M)c\cup ev^*(\omega)=0.
$$
So we have recover Corollary~\ref{differentiel dans SSS en cohomologie}.

Suppose that $i^*$ is onto and that $H^*(\Omega M)$ is torsion free,
the Serre spectral sequence collapses at the $E_2$-term.
So $\int i$ is injective. Therefore $\chi(M)f=0$.
And we have recover Corollary~\ref{tncz implique euler nulle}.
\section{TNCZ free loop fibrations}
Recall our result on TNCZ free loop fibration.
\begin{cor}(Corollary~\ref{tncz implique euler nulle})
Let $M$ be a connected, closed  $\Bbbk$-oriented manifold.
Suppose that the free loop fibration
$\Omega M\buildrel{i}\over\hookrightarrow LM\buildrel{ev}\over\twoheadrightarrow M$ is Totally Non-Cohomologous to Zero, i. e. $H^*(i)$ is onto
and that $H^k(\Omega M)$ is a torsion free ${\Bbbk}$-module
for each $k\geq 1$.
Then $\chi(M)=0$ in ${\Bbbk}$ or $M$ is a point.
\end{cor}
1) The first examples to have in mind are connected compact Lie groups.

2) Let $M$ be a sphere $S^d$ or the complex or quaternionic projective
space $\mathbb{CP}^n$, $\mathbb{HP}^n$.
Since $\chi(S^d)=1+(-1)^d$ and $\chi(\mathbb{CP}^n)=\chi(\mathbb{HP}^n)=n+1$,
it follows from our calculations in\cite{MenichiL:cohrfl}
that over any commutative ring ${\Bbbk}$,
$H^*(i;{\Bbbk})$ is onto if and only if $\chi(M)=0$ in ${\Bbbk}$
(when ${\Bbbk}$ is a field, this follows easily from the formality
of $M$ using the Jones isomorphism between Hochschild homology and free loop
space cohomology).

3) The converse of Corollary~\ref{tncz implique euler nulle} is not true:
if $n+1$ is not equal to $0$ in ${\Bbbk}$, take for example $M=\mathbb{CP}^n\times S^3$.

Over $\mathbb{Q}$, Vigu\'e-Poirrier has characterised which free loop fibrations
are TNCZ.
\begin{theor}\cite{VigueM:lacetslibresTNCZ}\label{tncz sur Q}
Let $X$ be a simply-connected topological space such that for all $n$,
$H^n(X;\mathbb{Q})$ is finite dimensional.
Then $$H^*(i;\mathbb{Q}):H^*(LX;\mathbb{Q})\rightarrow
H^*(\Omega X;\mathbb{Q})$$ is onto if and only if $H^*(X;\mathbb{Q})$
is a free graded commutative algebra.
\end{theor}
In~\cite[Theorem 2]{Kuribayashi:freeloopgrassmannstiefel},
Kuribayashi studied TNCZ free loop fibrations for some homogeneous spaces
over a prime field $\mathbb{F}_p$.

Let $V_k(\mathbb{R}^n)$ denotes the real Stiefel manifold of
orthonormal $k$-frames
in $\mathbb{R}^n$. Using the fibration $S^{n-k}\hookrightarrow V_k(\mathbb{R}^n)\twoheadrightarrow V_{k-1}(\mathbb{R}^n)$, we see that if $k\geq 2$, $\chi(V_k(\mathbb{R}^n))=0$.
Similarly for the complex or quartenionic Stiefel manifold,
$\chi(V_k(\mathbb{C}^n))=\chi(V_k(\mathbb{H}^n))=0$.
The Euler characteristic $\chi(G/H)$ of the quotient of a compact connected
Lie group $G$ by a connected closed subgroup $H$ of same rank is the quotient
$\vert W(G)\vert/\vert W(H)\vert$ of the cardinals of their Weyl groups~\cite[VII.Theorem 3.13]{Mimura-Toda:topliegroups}. Therefore $\chi(Sp(n)/U(n)=2^n$ 
and for Grassmannians $\chi(G_k(\mathbb{C}^n))=\chi(G_k(\mathbb{H}^n))=\binom {n}{k}$. We can now rewrite the main theorem of~\cite{Kuribayashi:freeloopgrassmannstiefel} in term of Euler characteristics.
\begin{theor}(\cite[Theorem 2]{Kuribayashi:freeloopgrassmannstiefel})

1) Let $M$ be $Sp(n)/U(n)$ or $V_k(\mathbb{C}^n))$ or $V_k(\mathbb{H}^n)$.
Then $H^*(i;\mathbb{F}_p):H^*(LM;\mathbb{F}_p)\rightarrow
H^*(\Omega M;\mathbb{F}_p)$ is onto if and only if $\chi(M)=0$ modulo $p$.

2) Let $M$ be $G_m(\mathbb{C}^{m+n})$ or $G_m(\mathbb{H}^{m+n})$
with $m$ and $n\geq 2$ and $p$ any prime.
Then $H^*(i;\mathbb{F}_p):H^*(LM;\mathbb{F}_p)\rightarrow
H^*(\Omega M;\mathbb{F}_p)$ is not onto.

3) Let $p$ an odd prime. Then
$H^*(i;\mathbb{F}_p):H^*(LV_m(\mathbb{R}^{m+n});\mathbb{F}_p)\rightarrow
H^*(\Omega V_m(\mathbb{R}^{m+n});\mathbb{F}_p)$
is onto if and only if $n$ is odd.
\end{theor}
So over $\mathbb{F}_p$, it is not clear when the converse of
Corollary~\ref{tncz implique euler nulle}
holds or not.

Let $X$ be a topological space.
Suppose that $\forall n\geq 0$, $H_n(\Omega X;\mathbb{Z})$ is a finitely
generated free abelian group, that $H_*(LX;\mathbb{Z})$ has no $p$-torsion
and that $H^*(i;\mathbb{F}_p):H^*(LX;\mathbb{F}_p)\rightarrow
H^*(\Omega X;\mathbb{F}_p)$ is onto. Then by the universal coefficient theorem
for homology, $H^*(i;\mathbb{Q}):H^*(LX;\mathbb{Q})\rightarrow
H^*(\Omega X;\mathbb{Q})$ is onto.

We now give a last result on TNCZ free loop fibration due to Iwase in
the context of classifying space $BG$ of finite loop spaces.
\begin{theor}~\cite[Theorem 2.2]{Iwase:adjoint}\label{TNCZ implique Pontryagin commutatif}
Let $X$ be a pointed topological space.
If $H_*(i;{\Bbbk}):H_*(\Omega X;{\Bbbk})\rightarrow H_*(LX;{\Bbbk})$
is injective then the Pontryagin algebra $H_*(\Omega X;{\Bbbk})$ is graded commutative (in particular $\pi_1(X)$ is abelian).
\end{theor}
\begin{proof}
We have the following two strictly commutative squares and the following
triangle commutting up to a homotopy $H$.
$$
\xymatrix{
\Omega X\times\Omega X\ar[r]^j\ar[d]_\mu
&LX\times_X LX\ar[d]^\mu\\
\Omega X\ar[r]_i
&LX
}
\xymatrix{
\Omega X\times\Omega X\ar[r]^j\ar[d]_\tau
&LX\times_X LX\ar[d]_\tau\ar[dr]^-\mu\\
\Omega X\times\Omega X\ar[r]_j
&LX\times_X LX\ar[r]_-\mu
&LX
}
$$
where the maps $\mu$ are the composition of loops and the maps $\tau$
are the exchange isomorphisms. 
The homotopy $H$ is the restriction to $[0,1/2]\times LX\times_X LX$
of the composite of
$$Id\times \mu:S^1\times LX\times_X LX\rightarrow S^1\times LX$$
and of the action of the circle on free loops $S^1\times LX\rightarrow LX$.
So finally, $i\circ\mu$ is homotopic to $i\circ\mu\circ\tau$.
Since $H_*(i)$ is injective, $H_*(\mu)\circ H_*(\tau)=H_*(\mu)$.
\end{proof}
Note that our homotopy between $i\circ\mu$ and $i\circ\mu\circ\tau$
is much simpler than the one arriving in $EG\times_G G^{ad}$ given by
Iwase in~\cite[Proof of Lemma 3.1]{Iwase:adjoint}.
\begin{ex}\label{exemple des suspensions}(Suspension)
Suppose that ${\Bbbk}$ is a field.
Let $X$ be a path-connected space such that
$H_*(X)$ is not concentrated in degree $0$.
By Bott-Samelson theorem, the Pontryagin algebra $H_*(\Omega \Sigma X)$
is isomorphic to the tensor algebra $TH_+(X)$ on the homology of $X$
in positive degrees.
Suppose that $H^*(i;{\Bbbk}):H^*(L\Sigma X;{\Bbbk})\rightarrow
H^*(\Omega \Sigma X;{\Bbbk})$ is surjective.
Then $H_*(i;{\Bbbk}):H_*(\Omega\Sigma X;{\Bbbk})\rightarrow
H_*(L\Sigma X;{\Bbbk})$ is injective. So by Theorem~\ref{TNCZ implique Pontryagin commutatif}, the Pontryagin ring $H_*(\Omega\Sigma X)$ is graded commutative.
So $H_+(X)$ is of dimension $1$ and is concentrated in even degree if the characteristic of ${\Bbbk}$ is different from $2$
(See~\cite[Example 2.6]{Kuribayashi:modulederivations} for a proof using Hochschild homology). In particular $\chi(\Sigma X)=0$ modulo the characteristic of ${\Bbbk}$.
\end{ex}
\begin{conjecture}\label{conjecturetncz}
Let $X$ be a simply-connected finite CW-complex and suppose that ${\Bbbk}$ is
a field.
If $H^*(i;{\Bbbk}):H^*(LX;{\Bbbk})\rightarrow
H^*(\Omega X;{\Bbbk})$ is onto then $\chi(X)$ is zero modulo the characteristic of ${\Bbbk}$ or $H_*(X)\cong {\Bbbk}$.
\end{conjecture}
It follows from the theorem of Vigue-Poirrier recalled above
(Theorem~\ref{tncz sur Q}) that the conjecture is true over the rationals.
In Example~\ref{exemple des suspensions}, we have checked the conjecture
for suspensions. In this paper, we proved the conjecture when
$X$ is a smooth connected, closed  $\Bbbk$-oriented manifold $M$
(Corollary~\ref{tncz implique euler nulle}).
We believe that conjecture~\ref{conjecturetncz} can be proved easily using Spanier-Whitehead duality.
\section{ the relative free loops case}
Let $g:N\rightarrow M$ be a map.
Let $g^*LM$ denote the {\em relative free loops space} of $g$ which is obtained by the following pull-back
$$\xymatrix{
g^*LM\ar[r]\ar[d]_p
&LM\ar[d]^{ev}\\
N\ar[r]_-{g}
&M.
}$$
\begin{theor}\label{formule relative free loop case}
Let $\sigma:N\hookrightarrow g^*LM$, $n\mapsto (n,\text{constant loop }g(n))$ be the section of the projection map
$p: g^*LM\rightarrow N$, $(n,w)\mapsto n$.
Let $M$ be a smooth  $h^*$-oriented manifold of dimension $m$.
Let $e_{TM}\in h^m(M)$ be the Euler class of the tangent bundle of $M$.
Suppose that $g$ is the composite $$N\buildrel{g_1}\over\twoheadrightarrow L\buildrel{g_2}\over\rightarrow M$$
where
$
\left\{
\begin{array}{l}
a)\; N  \mbox{  is a smooth manifold without boundary and }g_2\text{ is smooth and}\\
b)\; g_1\text{ is a (Serre) fibration}\,.
\end{array}\right.
$\\
Then for any $b\in h^*(g^*LM)$, 
$$
b\cup p^*\circ g^*(e_{TM})= p^*\circ \sigma^*(b)\cup p^*\circ g^*(e_{TM}).
$$
In particular if there is an integer $\text{dim }N$
such that $\forall i>\text{dim }N$, $h^i(N)=\{0\}$ then for all $b\in h^*(g^*LM)$ of degree $\vert b\vert>\text{dim }G-m$,
$b\cup p^*\circ g^*(e_{TM})=0$.
\end{theor}
\begin{ex}\label{exemple de fibration evaluation}
An interesting example is when the fibration $g=g_1:\text{map}(S,M)\twoheadrightarrow M$ is the evaluation at the base point of  a well-pointed space $S$.
In this case, the pull-back $g^*LM$ of $LM$ along $g$ is the space
$\text{map}(S\vee S^1,M)$ of maps from the wedge of $S$ and the circle to $M$.
\end{ex}
\begin{ex}
A generalization of the preceding example is when $g_2:L\rightarrow M$ is any smooth map and the fibration $g_1:\text{map}(S,L)\twoheadrightarrow L$ is the evaluation at the base point of  a well-pointed space $S$.
In this case, the pull-back $g^*LM$ of $LM$ along $g$ is the space
$\text{map}((S\vee S^1,S),(M,L))$ of couples of maps $(\varphi,\psi)$ such that the square
$\xymatrix{
S\ar[r]^\psi\ar[d]
&L\ar[d]^{g_2}\\
S\vee S^1\ar[r]^\varphi
&M.
}
$
commutes.
\end{ex}
\begin{cor}
Let $M$ be a smooth  $h^*$-oriented manifold of dimension $m$.
Let $e_{TM}\in h^m(M)$ be the Euler class of the tangent bundle of $M$.
Denote by $\vee_n S^1$ the wedge of $n\geq 0$ circles.
Let $\sigma_n:M\hookrightarrow \text{map}(\vee_n S^1,M)$, $m\mapsto\text{constant map }m$ be the section of the evaluation map
$ev_n: \text{map}(\vee_n S^1,M)\twoheadrightarrow M$.
Then for any $b\in h^*(\text{map}(\vee_n S^1,M))$, 
$$
b\cup ev_n^*(e_{TM})= ev_n^*\circ \sigma_n^*(b)\cup ev_n^*(e_{TM}).
$$
\end{cor}
\begin{proof}
When $n=0$, the formula is true since $ev_0$ and $\sigma_0$ are just the identity map of $M$.
By induction on $n$, suppose that for any $a\in h^*(\text{map}(\vee_{n-1} S^1,M))$, 
$$
a\cup ev_{n-1}^*(e_{TM})= ev_{n-1}^*\circ \sigma_{n-1}^*(a)\cup ev_{n-1}^*(e_{TM}).
$$
By example~\ref{exemple de fibration evaluation} in the case $S=\vee_{n-1} S^1$
and $g=ev_{n-1}$,
for any $b\in h^*(\text{map}(S\vee S^1,M))$,
$$
b\cup p^*\circ ev_{n-1}^*(e_{TM})= p^*\left(\sigma^*(b)\cup  ev_{n-1}^*(e_{TM})\right).$$
By taking $a=\sigma^*(b)$, the latter is equal to
$$p^*\left(ev_{n-1}^*\circ \sigma_{n-1}^*\circ\sigma^*(b)\cup ev_{n-1}^*(e_{TM})\right)
$$
Since $p^*\circ ev_{n-1}^*=ev_n^*$ and $\sigma_{n-1}^*\circ\sigma^*=\sigma_{n}^*$,
the conclusion follows.
\end{proof}
\begin{rem}
Again, let $S$ be a well-pointed space.
Suppose that the fibration
$\text{map}_*(S\vee S^1,M)\buildrel{i}\over\hookrightarrow \text{map}(S\vee S^1,M)\buildrel{ev}\over\twoheadrightarrow M$ is Totally Non-Cohomologous to Zero, i. e. $H^*(i)$ is onto.
Let $\pi:S\vee S^1\twoheadrightarrow S^1$ be the canonical projection.
Then we have the commutative triangle
$$
\xymatrix{
LM\ar[r]^-{\text{map}(\pi,M)}\ar[dr]_{ev}
& \text{map}(S\vee S^1,M)\ar[d]^{ev}\\
&M
}
$$
Since the induced map between the fibers in cohomology,
$H^*(\text{map}_*(\pi,M)):H^*(\text{map}_*(S\vee S^1,M))\rightarrow
H^*(\Omega M)$ is surjective,
the free loop fibration
$\Omega M\buildrel{i}\over\hookrightarrow LM\buildrel{ev}\over\twoheadrightarrow M$ is also Totally Non-Cohomologous to Zero, i. e. $H^*(i)$ is onto.

Conclusion: the preceding corollary together with Lemma~\ref{tncz implique classe d'euler nulle cohomologie singuliere}  is not really interesting to see if the fibration
$(\Omega M)^{\times n}\buildrel{i}\over\hookrightarrow\text{map}(\vee_n S^1,M)\buildrel{ev_n}\over\twoheadrightarrow M$ is Totally Non-Cohomologous to Zero or not.
\end{rem}

Let $f:N\rightarrow M$ and $g:N\rightarrow M$ be two maps.
Let $N\times_f M^I\times_g N$ denote the {\em homotopy coincidence point space} of $f$ and $g$ which is obtained by the following pull-back
$$
\xymatrix{
N\times_f M^I\times_g N\ar[r]\ar[d]_p
&M^I\ar[d]^{(ev_0,ev_1)}\\
N\ar[r]_-{(f,g)}
&M\times M.
}
$$
\begin{lem}\label{homotopy coincidence cas particulier}
Let $\bar{\xi}:N\times_f M^I\times_g N\hookrightarrow \leftidx{^{1,f}}{(N\times M)}{^{1,g}}$
be the map from the homotopy coincidence point space of $f$ and $g$ to the homotopy fibre product
of $(1,f):N\rightarrow N\times M$ and $(1,g):N\rightarrow N\times M$
defined by
$$\bar{\xi}(n,\omega)=(n,\text{the path } t\mapsto (n,\omega(t)),n)$$
for any $n\in N$ and any path $\omega:I\rightarrow M$ such that $\omega(0)=f(n)$ and $\omega(1)=g(n)$. Then $\bar{\xi}$ is a homotopy equivalence
\end{lem}
\begin{proof}
Consider the three pull-backs squares
$$
\xymatrix{
\leftidx{^{1,f}}{(N\times M)}{^{1,g}}\ar[r]\ar[d]_{(ev_0,ev_1)}
& (N\times M)^I\times_{1,g} N\ar[r]\ar[d]_{(ev_0,ev_1)}
&(N\times M)^I\ar[d]^{(ev_0,ev_1)}\\
N\times N\ar[r]_-{(1,f)\times 1}\ar[d]_{p_1}
&(N\times M)\times N\ar[r]_-{1\times (1,g)}\ar[d]^{p_1}
& (N\times M)\times (N\times M)\\
N\ar[r]_-{(1,f)}
&N\times M
}
$$
Here $p_1$ are the projections on the first factor.
Consider also the two pull-back squares
$$
\xymatrix{
N\times_f M^I\times_g N\ar[r]\ar[d]_{p}
& M^I\times_g N\ar[r]\ar[d]_{(p_N,ev_0)}
&M^I\ar[d]_{(ev_1,ev_0)}\\
N\ar[r]_-{(1,f)}
&N\times M\ar[r]_-{g\times 1}
&M\times M
}
$$
Let $\xi:M^I\times_g N\hookrightarrow (N\times M)^I\times_{1,g} N$
be the map 
defined by
$${\xi}(n,\omega)=(\text{the path } t\mapsto (n,\omega(t)),n)$$
for any $n\in N$ and any path $\omega:I\rightarrow M$ such that $\omega(1)=g(n)$.
Obviously ${\xi}$ is a homotopy equivalence.
We obtain the following commutative diagram where the two squares are pull-backs
according to the two previous diagrams.
$$
\xymatrix{
N\times_f M^I\times_g N\ar[r]\ar@/_4pc/[dd]_{p}\ar[d]^{\bar{\xi}}
& M^I\times_g N\ar@/^4pc/[dd]^{(p_N,ev_0)}\ar[d]^\xi_\approx\\
\leftidx{^{1,f}}{(N\times M)}{^{1,g}}\ar[r]\ar[d]_{ev_0}
& (N\times M)^I\times_{1,g} N\ar[d]_{ev_0}\\
N\ar[r]_-{(1,f)}
&N\times M
}
$$
By decomposing $(1,f)$ into the composite of a homotopy equivalence and of a fibration, we show using the structure of model category on topological spaces, that $\bar{\xi}$ is a homotopy equivalence since both $ev_0$ and $(p_N,ev_0)=ev_0\circ \xi$ are fibrations.
\end{proof}
\begin{proof}[Proof of Theorem~\ref{formule relative free loop case}]
Consider the two pull-back squares
$$
\xymatrix{
 N\ar[r]^-{(1,g)}\ar[d]_{g_1}
&N\times M\ar[d]^{g_1\times 1}\\
 L\ar[r]^-{(1,g_2)}\ar[d]_{g_2}
&L\times M\ar[d]^{g_2\times 1}\\
 M\ar[r]_-\Delta
&M\times M.
}
$$
Since $g_2$ and $1:M\rightarrow M$ are transverse, $g_2\times 1$
is transverse to the diagonal embedding $\Delta$. And so by Remark~\ref{pull-back transverse d'un embedding},
$(1,g_2):L\rightarrow L\times M$ is a proper embedding of codimension $m$ with $h^*$-oriented normal bundle.

Since $g_1$ is a (Serre) fibration, $g_1\times 1$ is also a (Serre) fibration.
Therefore $(1,g):N\rightarrow N\times M$ is the pull-back of an embedding
in the sense of definition~\ref{pull-back fibre ou transverse d'un embedding}.
So we can apply part 4) of Theorem~\ref{formule open string coproduct}.
Since the homotopy equivalence
$\bar{\xi}:N\times_f M^I\times_g N\buildrel{\approx}\over\hookrightarrow \leftidx{^{1,f}}{(N\times M)}{^{1,g}}$
of Lemma~\ref{homotopy coincidence cas particulier} in the case $f=g$
 commutes with the projection maps, i. e. $ev_0\circ \bar{\xi}=p$
 and also with the two sections $\sigma$,
 the ideal $\text{Ker } \sigma^*:h^*(g^*LM)\twoheadrightarrow h^*(N)$ satisfies
$$
\text{Ker } \sigma^*\cup p^*\circ g^*(e_{TM})=\{0\}.
$$
\end{proof}

\bibliography{Bibliographie}
\bibliographystyle{amsplain}
\end{document}